\documentclass[10pt]{amsart}
\usepackage{amsmath, amssymb, amsthm, graphicx}
\usepackage{epstopdf}

\numberwithin{equation}{section}

\theoremstyle{plain}
\newtheorem{prop}{Proposition}[section]
\newtheorem{conj}[prop]{Conjecture}
\newtheorem{coro}[prop]{Corollary}

\newtheorem{lemm}[prop]{Lemma}
\newtheorem{ques}[prop]{Question}
\newtheorem*{claim}{Claim}
\newtheorem*{prop1}{Proposition~1}
\newtheorem*{prop2}{Proposition~2}

\theoremstyle{definition}
\newtheorem*{defi}{Definition}
\newtheorem*{rema*}{Remark}
\newtheorem{rema}[prop]{Remark}
\newtheorem{exam}[prop]{Example}
\newtheorem*{nota}{Notation} 

\newcommand\area{a}
\newcommand\BOLD[1]{\pmb{\boldsymbol{#1}}}

\newcommand\CC{C}
\newcommand\compl{\mathrm{compl}_{\rev}}
\newcommand\concat{\mathbin{{}^{\scriptscriptstyle\frown}}}
\newcommand\DD{D}
\newcommand\DDD{\mathcal{D}}
\newcommand\dist{\mathrm{dist}}
\newcommand\ea{a}
\newcommand\eb{b}
\newcommand\ec{c}
\newcommand\ed{d}
\newcommand\ee{e}
\newcommand\ew{\epsilon}
\newcommand\ff{f}
\newcommand\g{\gamma}
\let\ge=\geqslant
\newcommand\ie{{\it i.e.}}
\newcommand\ii{i}
\newcommand\II{I}
\newcommand\Int[1]{[\![1,#1]\!]}

\newcommand\Inv{I}
\newcommand\jj{j}
\newcommand\kk{k}
\newcommand\KKK{\mathcal{K}}
\let\le=\leqslant
\newcommand\mm{m}
\newcommand\NF{{\scriptstyle\mathrm{NF}}}
\newcommand\Nba{N}
\newcommand\Ndc{N'}
\newcommand\nn{n}
\newcommand\NN{N}
\newcommand\nno{{n-1}}

\renewcommand\P[2]{\{#1,#2\}}
\newcommand\Ps[2]{\scriptstyle\{#1,#2\}}
\newcommand\pp{p}
\newcommand\ppp{p'}
\newcommand\Q[4]{\{\{#1,#2\},\{#3,#4\}\}}
\newcommand\Qs[4]{\scriptstyle\{\!\{#1,#2\},\{#3,#4\}\!\}}
\newcommand\qq{q}
\newcommand\qqq{q'}
\newcommand\resp{{\it resp.} }
\newcommand\rev{ \curvearrowright}
\newcommand\revk[1]{\rev^{\!#1}}
\newcommand\rr{r}

\newcommand\Sep{\Sigma}
\newcommand\SEP[2]{\Sigma_{#1,#2}}
\renewcommand\ss[1]{s_{#1}}
\newcommand\sss[1]{\overline{s}_{#1}}
\renewcommand{\SS}{S}
\newcommand\Sym{\mathfrak{S}}
\newcommand\sym[1]{\overline{#1}}
\newcommand\T[3]{\{#1,#2,#3\}}
\newcommand\Ts[3]{\scriptstyle\{#1,#2,#3\}}
\newcommand\uu{u}
\newcommand\vv{v}
\newcommand\VV{V}
\newcommand\ww{w}

\begin{document}

\hfill{\tiny 2009-02}

\author{Marc AUTORD}
\address{M.A., Laboratoire de Math\'ematiques Nicolas
Oresme, Universit\'e de Caen, 14032 Caen, France}
\email{marc.autord@math.unicaen.fr}

\author{Patrick DEHORNOY}
\address{P.D. \hbox{\rm(corresponding author)},
Laboratoire de Math\'ematiques Nicolas Oresme,
Universit\'e de Caen, 14032 Caen, France}
\email{dehornoy@math.unicaen.fr}
\urladdr{//www.math.unicaen.fr/\!\hbox{$\sim$}dehornoy}

\title{On the distance between the expressions of a permutation}

\keywords{permutation, transposition, reduced expression, combinatorial distance, van
Kampen diagram, subword reversing, braid relation}

\subjclass{20B30, 05E15, 20F55, 20F36}

\begin{abstract}
We prove that the combinatorial distance between any two reduced expressions of a
given permutation of~$\{1, ..., \nn\}$ in terms of transpositions lies in~$O(\nn^4)$, a sharp bound. Using a connection with the intersection numbers of certain curves in van Kampen diagrams, we prove that this bound is sharp, and give a practical criterion for
proving that the derivations provided by the reversing algorithm of [Dehornoy, JPAA 116
(1997) 115-197] are optimal. We also show the existence of length~$\ell$ expressions 
whose reversing requires $\CC \ell^4$ elementary steps.
\end{abstract}

\maketitle

This paper is about the various ways of expressing a
permutation as a product of transpositions and the
complexity of transforming one such expression into
another. We consider both the absolute complexity
(``combinatorial distance''), which deals with the minimal
possible number of steps, and the more specific complexity
(``reversing complexity''), which arises when one uses
subword reversing, a certain prescribed strategy for
transforming  expressions.

Throughout the paper, we denote by~$\Int\nn$ the set~$\{1, 2, ..., \nn\}$, and by~$\ss\ii$
the transposition that exchanges~$\ii$ and~$\ii+1$. A well known result---see for
instance  \cite{Hum}---states that, if $\pi$ is any permutation of~$\Int\nn$ and $\uu,
\vv$ are any two reduced (\ie, minimal length) expressions of~$\pi$  in terms of~$\ss1,
...,
\ss\nno$, then one can transform~$\uu$ into~$\vv$ only using the braid relations
\begin{align}
\label{E:Braid1}
\tag{I}
\ss\ii \ss\jj \ss\ii &= \ss\jj \ss\ii \ss\jj
\mbox{\,\quad with $\vert \ii - \jj \vert = 1$},\\
\label{E:Braid2}
\tag{II}
\ss\ii \ss\jj &= \ss\jj \ss\ii
\mbox{\qquad with $\vert \ii - \jj \vert \ge 2$}.
\end{align}
In this context, we define the \emph{combinatorial distance}~$\dist(\uu, \vv)$ of~$\uu$
and $\vv$ to be the minimal number of braid relations needed to transform~$\uu$
into~$\vv$. The standard proof for the finiteness of~$\dist(\uu,\vv)$ relies on the so-called
Exchange Lemma of Coxeter groups, and it leads to an exponential upper bound
for $\dist(\uu, \vv)$ in terms of~$\nn$. The first aim of this paper is to establish a
polynomial upper bound, namely a sharp degree~$4$ one. Using
\emph{``$\nn$-expression''} as a shorthand for ``expression representing a permutation
of~$\Int\nn$'', \ie, involving letters about~$\ss1, ..., \ss\nno$ only, we prove

\begin{prop1}
\label{P:Main}
There exist positive constants~$\CC_1, \CC_2$ such that, for each~$\nn$, 

$\bullet$ all equivalent reduced $\nn$-expressions~$\uu,
\vv$ satisfy $\dist(\uu,
\vv) \le \CC_1\,\nn^4$, 

$\bullet$ there exist equivalent $\nn$-expressions~$\uu, \vv$ satisfying
$\dist(\uu, \vv) \ge \CC_2\, \nn^4$.
\end{prop1}

(The values $\CC_1 = 1/2$ and $\CC_2 = 1/8$ are valid for~$\nn$ large enough.)

The methods we use are geometrical. For the upper bound, we consider some area in the
$\nn$-strand braid diagram naturally associated with an $\nn$-expression. For the lower
bound, we consider van Kampen diagrams and introduce certain curves called separatrices,
which are associated with the names of the strands involved in the successive
crossings. 

The latter notion, which seems of independent interest, provides general
criteria for proving that a van Kampen diagram (\ie, in algebraic terms, a derivation by
braid relations) is possibly optimal, \ie, it involves the minimal number of braid 
relations. In particular, Proposition~\ref{P:Minimal} below states that a
sufficient condition for a van Kampen diagram to be optimal is that any
two separatrices cross at most once in it.

In the second part, we address similar complexity issues in the particular case of subword
reversing. This is a specific strategy that, given two equivalent expressions~$\uu, \vv$,
returns a derivation of~$\vv$ from~$\uu$ by means of braid relations, \ie, equivalently,
constructs a van Kampen diagram for the pair~$(\uu, \vv)$. We observe on a simple
counter-example that the reversing method need not be optimal, but we deduce from the
above approach based on separatrices a simple optimality criterion, namely
Proposition~\ref{P:OptimalBis} that states that a sufficient condition for the
reversing method to be optimal for some pair~$(\uu, \vv)$ is that the so-called reversing
diagram for~$(\uu, \vv)$ contains no digon, \ie, in algebraic terms, the reversing
sequence from~$\sym\uu \vv$ contains no $\ew$-step---all technical terms are
defined below.

Finally, we address the general question of the complexity of the reversing method.
Frustratingly, the only upper bound we can establish at the moment is exponential---this
does not contradict the polynomial upper bound of Proposition~1,
since reversing need not be optimal. On the other hand, the
optimal lower bound of Proposition~1 induces a similar
lower bound in the case of reversing. What is more
interesting is to consider the case of non-necessarily
equivalent expressions. In that case, the reversing method
still applies, and its complexity remains widely unknown.
The relevant question is to determine the
number~$\compl(\uu,\vv)$ of elementary steps when one
starts with expressions~$\uu, \vv$ of length~$\ell$
(independently from the index~$\nn$). When the lower
bound of Proposition~1 is translated in this language, it
leads to a quadratic lower bound
$\compl(\uu,\vv) \ge \CC \, \ell^2$. This value is far from optimal.

\begin{prop2}
\label{P:Main}
There exists a positive constant~$\CC_3$ such that, for each~$\ell$,

$\bullet$ there exist length~$\ell$ expressions~$\uu, \vv$ satisfying
$\compl(\uu, \vv) \ge \CC_3\, \ell^4$.
\end{prop2}

(For $\ell$ large enough, we can take $\CC_3 =4/3$.) At the moment, we do not know
whether the above result is optimal.

It is likely that most results of this paper extend to finite Coxeter groups of other type.
However, the arguments developed here heavily rely on specific properties of permutations,
so how extending to Coxeter types other than~A and~B is not
clear.

\begin{rema*}
Most constructions developed in this paper in the case of permutations and
their  expressions in terms of transpositions can be extended to the case of positive braid
and their decompositions in terms of Artin's generators~$\sigma_\ii$. Technically, the case
of permutations corresponds to the particular case of the so-called simple braids,  which are
the divisors of Garside's fundamental braid~$\Delta_\nn$ in the braid
monoid~$B_\nn^+$, see~\cite{Gar} or \cite[Chap. 9]{Eps}.
Our reason for choosing the language of permutations here is
that it is more widely accessible and it avoids introducing
the  general framework of braids whereas most results would
involve simple braids exclusively. Indeed, it turns out that
the worst cases known so far, in particular in terms of
subword reversing, always involve simple braids.
We have no explanation for this phenomenon.
\end{rema*}

We finally mention that the results of Section~1
are mainly due to the second author, whereas those of
Section~2 are mainly due to the first author.

\section{The combinatorial distance}

Hereafter, a word~$\uu$ on the alphabet $\{\ss1, ..., \ss\nno\}$ is generically called an
\emph{$\nn$-expression}. Two $\nn$-expressions are called \emph{equivalent} if they
represent the same permutation of~$\Int\nn$. Throughout the paper (in particular in view
of the braid diagrams considered below), it is convenient that the product, both for words
and for permutations, refers to reverse composition: $\uu\vv$ means ``$\uu$ first,
then~$\vv$''. An $\nn$-expression~$\uu$ is called \emph{reduced} if the permutation
represented by~$\uu$ has no expression that is shorter than~$\uu$. 

If $\uu$ and $\vv$ are equivalent reduced $\nn$-expressions, then, as recalled above,
one can transform~$\uu$ into~$\vv$ using the braid relations
of types~I and~II, and
we denote by~$\dist(\uu,\vv)$ the minimal number of braid
relations needed to do it. In this section, we establish bounds
for~$\dist(\uu, \vv)$ when
$\nn$ grows to infinity.

\subsection{An upper bound result}
\label{S:Upper}

We begin with an upper bound result. To this end, we introduce one
distinguished reduced expression, called~\emph{normal}, for each permutation, and we
define a strategy that transforms any reduced expression into the (unique) normal
expression that represents the same braid.

For each $\nn$-expression~$\uu$, we define~$\DD_\uu$ to be the $\nn$-strand braid
diagram obtained by associating with the letter~$\ss\ii$ the pattern
\begin{equation}
\label{E:Diagram}
\vcenter{\hsize = 50mm
\begin{picture}(50,9)
\put(0,0){\includegraphics{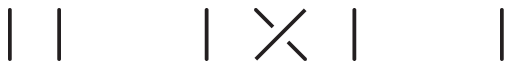}}
\put(-1,7){$1$}
\put(4,7){$2$}
\put(24,7){$\ii$}
\put(27,7){$\ii\!+\!1$}
\put(49,7){$\nn$}
\put(11,2){...}
\put(41,2){...}
\end{picture}}
\end{equation}
and by stacking from top to bottom the elementary patterns corresponding to the
successive letters of~$\uu$. When we speak of the \emph{$\pp$th strand} in~$\DD_\uu$,
we refer to the strand that starts at the $\pp$th position from the left on the top line. It is
well known that $\uu$ is a reduced expression if and only if any two strands
in~$\DD_\uu$ cross at most once~\cite[Chap{.}\,9]{Eps}.

\begin{defi}
Define $\ss{\jj, \ii}$ to be $\ss{\jj-1} \ss{\jj-2} ... \ss{\ii+1} \ss\ii$ for $\jj > \ii$,
and to be the empty word~$\ew$ for $\jj = \ii$. An $\nn$-expression is called
\emph{normal} if it has the form
\begin{equation}
\label{E:Normal}
\ss{1, \ff(1)} \, \ss{2, \ff(2)} \, ... \,  \ss{\nn, \ff(\nn)},
\end{equation}
 for some function~$\ff : \Int\nn \to \Int\nn$ satisfying $\ff(\ii)
\le \ii$ for each~$\ii$.
\end{defi} 

(The first factor $\ss{1, \ff(1)}$ is mentioned for symmetry, but is necessarily empty.)

\begin{lemm}
$(i)$ Every normal expression is reduced.

$(ii)$ Each permutation of~$\Int\nn$ admits a unique normal $\nn$-expression.
\end{lemm}

\begin{proof}
$(i)$ Assume that $\uu$ has the form~\eqref{E:Normal}. Then, for each~$\ii$, the
$\ii$th strand crosses over no $\jj$th strand with $\jj > \ii$ in the diagram~$\DD_\uu$.
Therefore, any two strands cross at most once in~$\DD_\uu$, so $\uu$ is a reduced
expression.

$(ii)$ For $\ff$ satisfying $\ff(\ii) \le \ii$ for each~$\ii$, let $\nn_\ff$ be the
largest~$\mm$ satisfying $\ff(\mm) < \mm$, if any, and $0$ otherwise, and  let $\pi_\ff$
be the permutation represented by the expression~\eqref{E:Normal} associated
with~$\ff$. We prove that $\pi_\ff$ determines~$\ff$ using induction on~$\nn_\ff$.
For $\nn_\ff = 0$, the only possibility is
that $\ff(\ii) = \ii$ holds for each~$\ii$, so the permutation~$\pi_\ff$ is the identity.
Assume now $\nn_\ff \ge 1$. By construction, we have $\pi_\ff(\ii) = \ii$ for $\ii \ge
\nn_\ff$, and
$\pi_\ff(\nn_\ff) = \ff(\nn_\ff)$. Hence $\pi_\ff$ determines~$\nn_\ff$
and~$\ff(\nn_\ff)$. Next, let $\ff'$ be defined by
$\ff'(\ii) = \ff(\ii)$ for $\ii < \nn_\ff$, and
$\ff'(\ii) = \ii$ for $\ii \ge \nn_\ff$. Then we have $\pi_{\ff'}  = \pi_\ff \ss{\nn_\ff,
\ff(\nn_\ff)}$, hence $\pi_\ff$ determines~$\pi_{\ff'}$. By construction, we have
$\nn_{\ff'} \le \nn_\ff - 1$. By induction hypothesis, $\pi_{\ff'}$ determines~$\ff'$,
hence so does~$\pi_\ff$. Finally, $\ff$ is determined by~$\ff', \nn_\ff$, and
$\ff(\nn_\ff)$, hence by~$\pi_\ff$.
\end{proof}

For each (reduced) expression~$\uu$, we denote by~$\NF(\uu)$ the unique normal
expression that is equivalent to~$\uu$. We shall now define a strategy for
transforming~$\uu$ into~$\NF(\uu)$. 

First, we concentrate on the last factor of the normal form. We have associated with every
$\nn$-expression~$\uu$ a braid diagram~$\DD_\uu$. We shall assume that the
pattern~\eqref{E:Diagram} is drawn in a rectangle that has width~$\nno$ and height~$1$.
So, if $\uu$ is an $\nn$-expression of length~$\ell$, the diagram~$\DD_\uu$ is
drawn in an $(\nno) \times \ell$ grid. It includes $(\nno) \ell$ squares of
size~$1$, and it makes sense to count how many such squares lie on the left or on the
right of a given strand.

\begin{lemm}
\label{L:Area}
For each $\nn$-expression~$\uu$, define $\area(\uu)$ to be the number of plain
squares lying on the right of the $\nn$th strand in the diagram~$\DD_\uu$. Then, for
each reduced $\nn$-expression~$\uu$, there exists an equivalent reduced
expression~$\vv \ss{\nn, \kk}$, with $\vv$ an  $(\nno)$-expression, satisfying
\begin{equation}
\label{E:Area}
\dist(\uu, \vv \ss{\nn,\kk}) \le \area(\uu).
\end{equation}
\end{lemm}

For an induction it is enough to establish

\begin{claim}
If $\uu$ is not of the form $\vv \ss{\nn,\kk}$ with $\vv$ an $(\nno)$-expression, there
exists an $\nn$-expression~$\uu'$ satisfying
$\dist(\uu, \uu') = 1$ and  $\area(\uu') < \area(\uu)$.
\end{claim}

\begin{proof}
Let $\pp$ be the final position of the $\nn$th strand in~$\DD_\uu$. The hypothesis
implies  $\pp < \nn$ holds as, otherwise, $\uu$ itself would be an $(\nno$)expression and
it could be expressed as $\uu\ss{\nn,\nn}$. Then there exists a unique decomposition
$$\uu = \vv \, \ss{\nn, \ii} \, \ss\jj \, \ww$$
with $\vv$ an $(\nno)$-expression, $\ii < \nn$, and $\jj \not= \ii-1$: we consider the
first block of crossings $\ss{\nn, \ii}$ in which the $\nn$th strand is the front strand,
and the hypothesis on~$\uu$ means that, after that block, there still remains at least one
crossing~$\ss\jj$ in which the $\nn$th strand is not the front strand, which means $\jj
\not= \ii-1$. We consider the various possible values of~$\jj$ with respect to~$\ii$. First,
$\jj = \ii -1$ is excluded by hypothesis, whereas $\jj = \ii$ would contradicts the
hypothesis that $\uu$ is reduced since the $\nn$th strand would cross the same strand
twice. There remain two cases only.

\noindent
Case 1: $\vert \jj - \ii \vert \ge 2$. Put $\uu' = \vv \ss{\nn, \ii+1} \ss\jj
\ss\ii \ww$. Then we have $\dist(\uu, \uu')\nobreak =\nobreak 1$, and
$\area(\uu') =
\area(\uu) - 1$, as shown in the following diagrams, which compare the contributions of
the factors $\ss\ii
\ss\jj$ and $\ss\jj \ss\ii$ to the right hand side area of the $\nn$th strand (the $\nn$th
strand is in bold, and the squares contributing to~$\area$ are in grey)
\begin{center}
\begin{picture}(60,25)
\put(0,0){\includegraphics{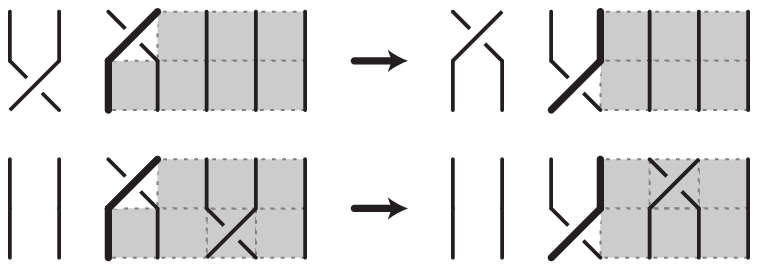}}
\put(-30,19){case $\jj \le \ii - 2$:}
\put(-30,4){case $\jj \ge \ii + 2$:}
\put(19,-3){$\jj$}
\put(9,-3){$\ii$}
\put(-1,12){$\jj$}
\put(9,12){$\ii$}
\end{picture}
\end{center}

\noindent
Case 2: $\jj = \ii + 1$. By construction this may happen only for $\ii \le \nn- 2$. Let $\uu'
= \vv \ss{\nn, \ii+2} \ss\ii \ss\jj \ss\ii \ww$. Then we have again $\dist(\uu,
\uu')\nobreak =\nobreak 1$, and $\area(\uu') =
\area(\uu) - 2$, as shown in the diagram
\begin{center}
\begin{picture}(60,15)
\put(0,0){\includegraphics{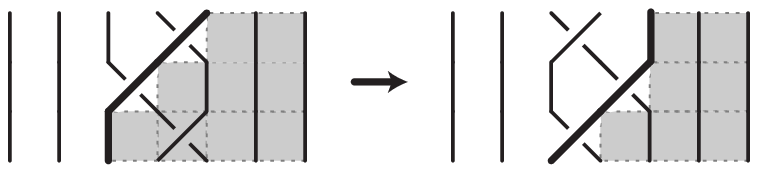}}
\put(-30,7){case $\jj = \ii + 1$:}
\put(14,-3){$\jj$}
\put(9,-3){$\ii$}
\end{picture}
\end{center}

So the proof of the claim is complete, and the lemma follows.
\end{proof}

Repeated uses of Lemma~\ref{L:Area} lead to

\begin{lemm}
\label{L:NF}
For every reduced $\nn$-expression~$\uu$ of length~$\ell$, we have
\begin{equation}
\label{E:NF}
\dist(\uu, \NF(\uu)) \le \nn(\nno) \ell / 2.
\end{equation}
\end{lemm}

\begin{proof}
We use induction on~$\nn$. The result is obvious for $\nn \le 3$. Assume $\nn \ge 4$.
Assume that the last factor in~$\NF(\uu)$ is $\ss{\nn, \kk}$. By Lemma~\ref{L:Area},
there exists a reduced $(\nno)$-expression~$\vv$ of length at most~$\ell$---actually
of length exactly $\ell - (\nn -
\kk)$---satisfying
$\dist(\uu, \vv \ss{\nn, \kk}) \le \area(\uu)$. By construction, we have
$\area(\uu) \le (\nn-1)\ell$. On the other hand, the uniqueness of the normal form implies
$\NF(\uu)  =
\NF(\vv)
\ss{\nn, \kk}$. Hence, using the induction hypothesis, we find
$$\dist(\uu, \NF(\uu)) \le (\nno)\ell + (\nn-1)(\nn-2)\ell/2 = \nn(\nno)\ell/2.
\eqno{\square}$$
\def\qed{\relax}
\end{proof}

We immediately deduce

\begin{prop}
\label{P:Upper}
For all equivalent $\nn$-expressions~$\uu, \vv$ of length~$\ell$, we have
\begin{equation}
\label{E:Dist}
\dist(\uu, \vv) \le (\nn-1)(\nn-2) \ell.
\end{equation}
\end{prop}

As every reduced $\nn$-expression has length at most $\nn(\nn-1)/2$, the upper
bound $O(\nn^4)$ of Proposition~1 follows.

\begin{rema}
Considering the rightmost strand in the argument of Lemma~\ref{L:Area} is essential.
Indeed, for each reduced expression~$\uu$ and each~$\ii$ such that $\ss\ii \uu$ is not
reduced, \ie, such that the $\ii$th and the $\ii+1$st strands cross in~$\DD_\uu$, we can
consider the area~$\area_\ii$ of the domain bounded by the top line and the $\ii$th and
$\ii+1$st strands before they cross. It is natural to wonder whether $\uu$ can be
transformed into an equivalent expression~$\ss\ii \vv$ in such a way that the
parameter~$\area_\ii$ decreases at each step---thus obtaining a new proof of the Exchange
Lemma. The answer is negative. Indeed, assume~$\uu = \ss1\ss3\ss2\ss1\ss3 \ss2$. Then
the second and third strands cross in~$\DD_\uu$ and $\ss2 \uu$ is not reduced: $\uu$
is equivalent to $\ss2 \ss3\ss2\ss1\ss2\ss3$. However, there is no way to apply a
braid relation to~$\uu$ so as to decrease the area~$\area_2$ of the domain bounded by the
second and third strands. Indeed, the two expressions at distance~$1$ from~$\uu$ are
$\uu' =
\ss3\ss1\ss2\ss1\ss3
\ss2$ and $\uu'' =
\ss1\ss3\ss2\ss3\ss1
\ss2$, which satisfy $\area_2(\uu)  = \area_2(\uu')  = \area_2(\uu'')  = 9$, as shown in
the diagrams
\begin{center}
\begin{picture}(78,32)
\put(0,0){\includegraphics{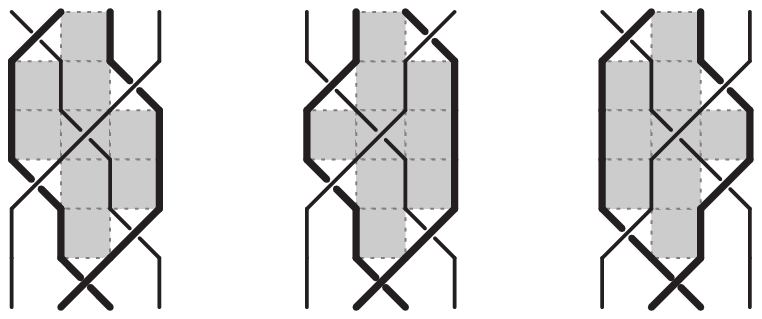}}
\put(-4,28){$\uu$}
\put(25,28){$\uu'$}
\put(54,28){$\uu''$}
\end{picture}
\end{center}
in each of which nine grey squares occur.
\end{rema}

\subsection{A lower bound result}
\label{S:Lower}

We turn to the other direction, namely proving lower bounds on the combinatorial
distance of two equivalent reduced expressions. To this end, we associate a name to each
letter in a reduced expression and observe that applying one braid relation can only change
the associated sequence of names by a limited amount.

\begin{nota}
Hereafter, we use $\Int\nn^{(\kk)}$ for the set of all subsets of~$\Int\nn$ that have
cardinality~$\kk$ (exactly), and $\Int\nn^{(2,2)}$ for the set of all subsets
of~$\Int\nn^{(2)}$ that have cardinality~$2$, \ie, the set of all non-degenerate pairs
of pairs in~$\Int\nn$.
\end{nota}

By construction, every crossing in a braid diagram~$\DD_\uu$ involves two strands,
each of which has an initial position that corresponds to an integer in~$\Int\nn$. By
considering the initial positions of the strands that cross there, we associate with each
instance of a letter~$\ss\ii$ in an $\nn$-expression a well defined pair~$\P\pp\qq$
in~$\Int\nn^{(2)}$, hereafter called its \emph{name}.

\begin{defi}
For each reduced $\nn$-expression~$\uu$, we define $\SS(\uu)$ to be the sequence
composed of the names of the successive letters in~$\uu$. 
\end{defi}

So, formally, $\SS(\uu)$ is the sequence in~$\Int\nn^{(2)}$ recursively defined by
$\SS(\ew) = \emptyset$ (the empty sequence) and, using
$\concat$ for concatenation, 
$$\SS(\uu) = \SS(\vv) \concat (\{\pp, \qq\}),$$
assuming that $\uu = \vv \ss\ii$ and the strands that finish at positions~$\ii$
and~$\ii+1$ in~$\DD_\vv$ are the $\pp$th and the $\qq$th ones, \ie, start at
positions~$\pp$ and~$\qq$, respectively.

\begin{exam}
\label{X:Flip}
Let $\uu_\nn = \ss{1, 1} \ss{2, 1} ... \ss{\nn, 1}$. Then $\uu_\nn$ is a reduced
expression of the Coxeter element of~$\Sym_\nn$, \ie, of the flip permutation~$\phi$
that exchanges $\ii$ and $\nn - \ii$ for each~$\ii$. An easy induction gives
\begin{equation}
\label{E:Flip11}
\SS(\uu_\nn) = (\{1,2\}, \{1, 3\}, \{2,3\}, ..., \{\nn-2,\nno\} \{1,\nn\}, \{2,\nn\}, ...,
\{\nno,\nn\}).
\end{equation}
Symmetrically, let $\vv_\nn$ be the expression obtained from~$\uu_\nn$ by reversing
the order of the factors and flipping their entries, \ie,
$\vv_\nn = \ss{\nn,1} \ss{\nn,2}  ... \ss{\nn, \nno}$. Then
$\vv_\nn$ is another reduced expression of~$\phi$, and we find
\begin{equation}
\label{E:Flip2}
\SS(\vv_\nn) = (\{\nno,\nn\}, ..., \{2,\nn\}, \{1, \nn\}, \{\nn-2,\nno\}, ..., \{2,3\}, \{1,
3\}, \{1,2\}):
\end{equation}
so $\SS(\vv_\nn)$ is the sequence obtained by reversing the entries of~$\SS(\uu_\nn)$.
\end{exam}

By construction, if $\uu$ is a reduced $\nn$-expression, the pair~$\{\pp,\qq\}$ occurs
in~$\SS(\uu)$ if and only if the strands starting at positions~$\pp$ and~$\qq$ cross in
the diagram~$\DD_\uu$, hence if and only if $\P\pp\qq$ is an inversion of the
permutation represented by~$\uu$. Hence, if $\uu, \vv$ are equivalent reduced
$\nn$-expressions, the pairs occurring in~$\SS(\uu)$ and in~$\SS(\vv)$ coincide, and
$\SS(\vv)$ is a permuted image of~$\SS(\uu)$. 

We shall see now that comparing the
sequences~$\SS(\uu)$ and~$\SS(\vv)$ leads to a lower bound on the combinatorial
distance between~$\uu$ and~$\vv$. 

\begin{defi}
If $\SS, \SS'$ are enumerations of $\Int\nn^{(2)}$, we denote by~$\II_3(\SS, \SS')$
(\resp $\II_{2,2}(\SS, \SS')$) the number of triples~$\T\pp\qq\rr$ in~$\Int\nn^{(3)}$
(\resp the number of pairs of pairs $\{\{\pp,\qq\},\{\ppp,\qqq\}\}$
in~$\Int\nn^{(2,2)}$) such that the order of the pairs
$\{\pp,\qq\}$, $\{\pp,
\rr\}$, $\{\qq,\rr\}$ (\resp the order of $\P\pp\qq$ and $\P\ppp\qqq$) is not the
same in~$\SS$ and~$\SS'$. 
\end{defi}

\begin{prop}
\label{P:Lower}
For all equivalent $\nn$-expressions~$\uu, \vv$, we have
\begin{equation}
\label{E:Lower}
\dist(\uu, \vv) \ge \II_3(\SS(\uu), \SS(\vv)) + \II_{2,2}(\SS(\uu), \SS(\vv)).
\end{equation}
More precisely, every derivation from~$\uu$ to~$\vv$ contains at least $\II_3(\SS(\uu),
\SS(\vv))$ relations of type~I and $\II_{2,2}(\SS(\uu),
\SS(\vv))$ relations of type~II.
\end{prop}

\begin{proof}
We first consider the case $\dist(\uu, \vv) = 1$, \ie, the case when $\vv$ is obtained
from~$\uu$ by applying one braid relation. Assume first that $\vv$ is obtained by
applying a type~I relation $\ss\ii \ss{\ii+1} \ss\ii = \ss{\ii+1} \ss\ii
\ss{\ii+1}$. Then there exists a unique triple $\T\pp\qq\rr$, namely the names of
the three strands involved in the transformation, such that the sequence~$\SS(\vv)$ is
obtained from~$\SS(\uu)$ by replacing the subsequence $\{\pp, \qq\}, \{\pp, \rr\},
\{\qq, \rr\}$ with $\{\qq, \rr\}, \{\pp, \rr\},\{\pp, \qq\}$---see Figure~\ref{F:Names}
below for an illustration. So the order of the three pairs arising from~$\T\pp\qq\rr$ has
changed between~$\SS(\uu)$ and~$\SS(\vv)$. On the other hand, any other triple 
in~$\Int\nn^{(3)}$ has at most two elements in common with
$\T\pp\qq\rr$, and the order of the three pairs arising from that triple is the same
in~$\SS(\uu)$ and~$\SS(\vv)$. Hence, we have $\II_3(\SS(\uu), \SS(\vv)) = 1$ in this case.
On the other hand, any pair of pairs in~$\Int\nn^{(2,2)}$ contains at most one of the
three pairs $\{\pp,\qq\}$, $\{\pp, \rr\}$, $\{\qq,\rr\}$ and, therefore, the order of this
pair is not changed from~$\SS(\uu)$ and~$\SS(\vv)$. Hence, we have $\II_{2,2}(\SS(\uu),
\SS(\vv)) = 0$ in this case.

Assume now that $\vv$ is obtained by applying a 
type~II relation $\ss\ii
\ss\jj =
\ss\jj \ss\ii$ with $\vert\jj-\ii\vert \ge 2$. Then there exists a unique pair of pairs
$\{\pp,\qq\}$, $\{\ppp, \qqq\}$ in~$\Int\nn^{(2,2)}$, namely the names of the four
strands involved in the transformation, such that the sequence~$\SS(\vv)$ is obtained
from~$\SS(\uu)$ by replacing the subsequence $\{\pp, \qq\}, \{\ppp, \qqq\}$ with
$\{\ppp, \qqq\},\{\pp, \qq\}$. So the order of the two considered pairs has changed
between~$\SS(\uu)$ and~$\SS(\vv)$. On the other hand, any other pair of pairs
in~$\Int\nn^{(2,2)}$ contains at most one of the two pairs $\{\pp, \qq\}, \{\ppp, \qqq\}$,
and its order is the same in~$\SS(\uu)$ and~$\SS(\vv)$. Hence, we have
$\II_{2,2}(\SS(\uu), \SS(\vv)) = 1$. Moreover, any triple in~$\Int\nn^{(3)}$ gives rise
to a triple of pairs that contains at most one of $\{\pp, \qq\}, \{\ppp, \qqq\}$. Hence
the order of the three pairs in this triple has not changed from~$\SS(\uu)$
and~$\SS(\vv)$, and we have $\II_3(\SS(\uu), \SS(\vv)) = 0$.

We conclude that, in every case, the quantity $\II_3(\SS(\uu), \SS(\vv)) + \II_{2,2}(\SS(\uu),
\SS(\vv))$ changes by not more than one (and even by exactly one) when one braid
relation is applied. This clearly implies~\eqref{E:Lower}.
\end{proof}

We can now complete the proof  of Proposition~1.

\begin{coro}
\label{C:Lower}
For each~$\nn$, there exist reduced $\nn$-expressions~$\uu, \vv$ satisfying
\begin{equation}
\label{E:Example}
\dist(\uu, \vv) \ge \frac18 \nn^4 + O(\nn^3).
\end{equation}
\end{coro}

\begin{proof}
Consider the expressions~$\uu_\nn, \vv_\nn$ of Example~\ref{X:Flip}. As observed
above, the sequences $\SS(\uu_\nn)$ and~$\SS(\vv_\nn)$ are mirror images of one
another. It follows that each triple~$\T\pp\qq\rr$ in~$\Int\nn^{(3)}$ contributes~$1$ to
the parameter~$\II_3$, leading to
\begin{equation*}
\II_3(\SS(\uu_\nn), \SS(\vv_\nn)) = \# (\Int\nn^{(3)}) = {\nn \choose 3}.
\end{equation*}
Similarly, each pair of pairs
$\{\{\pp, \qq\}, \{\ppp, \qqq\}\}$ in~$\Int\nn^{(2,2)}$ contributes~$1$
to the parameter~$\II_{2,2}$, giving
$$\II_{2,2}(\SS(\uu_\nn), \SS(\vv_\nn)) = \# (\Int\nn^{(2,2)}) =
\frac12{\nn\choose2}{\nn-2 \choose 2} = 3{\nn
\choose 4}.
\eqno\square
$$
\let\qed\relax
\end{proof}

Owing to the last sentence in Proposition~\ref{P:Lower}, we 
also deduce from the above computation the result that the
number of type~I braid relations occurring
in every derivation of~$\vv$ from~$\uu$ is at least
$\nn^3/6 + O(\nn^2)$.

In the context of Proposition~\ref{P:Lower}, it is natural
to wonder whether \eqref{E:Lower} is always an
equality, as it turns out to be in simple cases.

\begin{ques}
Does the equality 
\begin{equation}
\label{Q:Lower}
\dist(\uu, \vv) = \II_3(\SS(\uu), \SS(\vv)) +
\II_{2,2}(\SS(\uu), \SS(\vv))
\end{equation}
hold for all equivalent 
$\nn$-expressions~$\uu, \vv$?
\end{ques}

So far we have not been able to obtain any answer. In
particular, computer tries failed to find expressions
disproving the equality~\eqref{Q:Lower}.

\begin{rema}
A result similar to Proposition~\ref{P:Lower} can be obtained by simply counting the
inversion number~$\II(\SS(\uu), \SS(\vv))$ of the sequences~$\SS(\uu)$ and~$\SS(\vv)$,
\ie, the total number of pairs of pairs whose order is changed. One easily checks that
~$\II(\SS(\uu), \SS(\vv))$ is changed by three when a type~I relation is
applied, and by~one in the case of a type~II relation, thus leading
to\begin{equation*}
\label{E:LowerBis}
\dist(\uu, \vv) \ge \Inv(\SS(\uu), \SS(\vv))/3.
\end{equation*}

Also, it can be mentionned that using areas in braid diagrams as in
Section~\ref{S:Upper}  can also lead to lower bounds on the
combinatorial distance. Indeed, it is easy to check that applying one braid relation can
change such areas by a bounded factor~$K$ only, leading to inequalities of the generic form 
$$\dist(\uu, \vv) \ge \vert \mathrm{area}(\DD_\uu) -
\mathrm{area}(\DD_\vv)\vert/ K.$$
\end{rema}

\subsection{Van Kampen diagrams}
\label{S:Van}

Proposition~\ref{P:Lower} and the approach of Section~\ref{S:Lower} leads to a nice
geometric criterion for proving that a derivation between two reduced expressions~$\uu,
\vv$ of a permutation~$\pi$ is possibly optimal, \ie, that it realizes the combinatorial
distance~$\dist(\uu,\vv)$ between~$\uu$ and~$\vv$.

Assume that $\uu, \vv$ are equivalent expressions. A \emph{van Kampen
diagram for~$(\uu, \vv)$} is a planar connected diagram~$\KKK$ consisting of
finitely many adjacent tiles of the two types 
$$\begin{picture}(114,17)(0,0)
\put(0,7){type~I:}
\put(13,0){\includegraphics{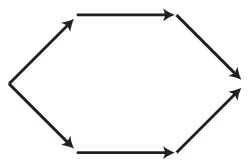}}
\put(13,2){$\ss\jj$}
\put(13,12){$\ss\ii$}
\put(23,-2){$\ss\ii$}
\put(23,16){$\ss\jj$}
\put(34,2){$\ss\jj$}
\put(34,12){$\ss\ii$}
\put(38,7){with $\vert\ii-\jj\vert{=}1$,}
\put(64,7){type II:}
\put(77,0){\includegraphics{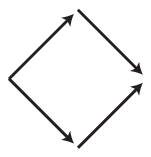}}
\put(77,2){$\ss\jj$}
\put(77,12){$\ss\ii$}
\put(88,12){$\ss\jj$}
\put(88,2){$\ss\ii$}
\put(92,7){with $\vert\ii-\jj\vert{\ge}2$,}
\end{picture}$$
and such that the boundary
of~$\KKK$ consists of two paths labeled~$\uu$ and~$\vv$. Because of the orientation of the
edges of the tiles, such a diagram has exactly one initial vertex (source) and one terminal
vertex (sink), and two boundary paths labeled~$\uu$ and $\vv$ from the
source to the sink. 

It is standard---see for
instance~\cite{LyS} or~\cite{Eps}---that, if
$\uu, \vv$ are reduced expressions, then $\vv$ can be derived from~$\uu$ using braid
relations if and only if there exists a van Kampen diagram for~$(\uu, \vv)$. See
Figure~\ref{F:Van} for an example. More precisely, if $\vv$ can be derived from~$\uu$
using $\NN$~braid relations, then there exists a van Kampen diagram for~$(\uu, \vv)$ that
contains $\NN$~tiles, and conversely. So, we have the natural notion of an optimal (or
minimal) van Kampen diagram:

\begin{defi}
Assume that $\KKK$ is a van Kampen diagram for~$(\uu, \vv)$. We say that $\KKK$ is
\emph{optimal} if the number of tiles in~$\KKK$ equals the combinatorial
distance~$\dist(\uu, \vv)$.
\end{defi}

In other words, a van Kampen diagram is declared optimal if there exists no smaller
(\ie, containing less tiles) diagram with the same boundary.
Our aim in the sequel will be to describe criteria for
recognizing that a van Kampen diagram is possibly optimal.
The first idea is to use names. Assume that $\KKK$ is a van Kampen diagram for a
pair~$(\uu, \vv)$. Then we can unambiguously attribute a \emph{name} with every tile
of~$\KKK$. First, we attribute a name to each edge of~$\KKK$. Let $\ee$ be such an edge. It
follows from the definition of a van Kampen diagram that there exists at least one path that
connects the source of~$\KKK$ to any vertex~$\VV$, and at least one path that
connects~$\VV$ to the sink of~$\KKK$. Hence there is a path~$\g$ that connects the source
of~$\KKK$ to its sink and contains~$\ee$. Then the successive labels of the edges of~$\g$
make an expression~$\ww$, which is certainly equivalent to~$\uu$ (and~$\vv$) since, by
construction, there is a van Kampen diagram for~$(\uu, \ww)$, namely the subdiagram
of~$\KKK$ bounded by~$\g$ and the path labeled~$\uu$. Then, as in
Section~\ref{S:Lower}, we attribute a name to each letter of~$\ww$, and copy these names
on the corresponding edges of~$\g$. In this way, we have given a name to~$\ee$ which is a
certain pair in~$\Int\nn^{(2)}$. We claim that this name only depends on~$\ee$, and not
on the choice of the path~$\g$. As any two paths~$\g, \g'$ correspond to equivalent
expressions~$\ww, \ww'$, it is enough to consider the case when $\ww$ and $\ww'$ are
deduced from one another using one braid relation, and the result is then obvious. Then, we
observe that the names occurring on the edges of an elementary tile can be of two types
only, namely those displayed in Figure~\ref{F:Names}.

\begin{figure}[htb]
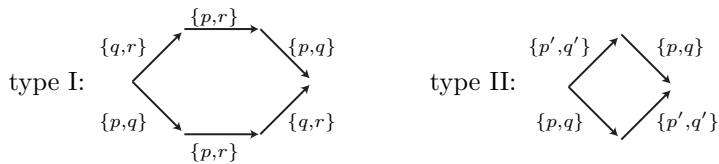

\begin{picture}(98,21)(-12,0)
\put(-12,7){type~I:}
\put(4,0){\includegraphics{TileI.eps}}
\put(0,2){$\Ps\pp\qq$}
\put(0,12){$\Ps\qq\rr$}
\put(25,12){$\Ps\pp\qq$}
\put(25,2){$\Ps\qq\rr$}
\put(12,-2){$\Ps\pp\rr$}
\put(12,16){$\Ps\pp\rr$}
\put(44,7){type~II:}
\put(62,0){\includegraphics{TileII.eps}}
\put(58,2){$\Ps\pp\qq$}
\put(57,12){$\Ps\ppp\qqq$}
\put(74,12){$\Ps\pp\qq$}
\put(74,2){$\Ps\ppp\qqq$}
\end{picture}
\caption{\sf Giving names to the tiles in a van Kampen diagram: the left tile is called
$\T\pp\qq\rr$, the right one is called~$\Q\pp\qq\ppp\qqq$.}
\label{F:Names}
\end{figure}

\begin{defi}
(See Figures~\ref{F:Names} and~\ref{F:Van}.) Assume that $\KKK$ is a van Kampen diagram
for~$(\uu,\vv)$.  We define the \emph{name} of a tile in~$\KKK$ as follows:

$\bullet$ for a type~I tile, it is the (unique) triple~$\T\pp\qq\rr$ of~$\Int\nn^{(3)}$ such
that the names of the border edges are $\{\pp, \qq\}$,
$\{\pp, \rr\}$, and $\{\qq, \rr\}$;  

$\bullet$ for a type~II tile, it is the (unique) pair~$\{\{\pp, \qq\},
\{\ppp, \qqq\}\}$ of~$\Int\nn^{(2,2)}$ such that the names
of the border edges are $\{\pp,
\qq\}$ and
$\{\ppp, \qqq\}$.  
\end{defi}

\begin{figure}[htb]
\begin{picture}(88,50)(0,1)
\put(-1,0.5){\includegraphics{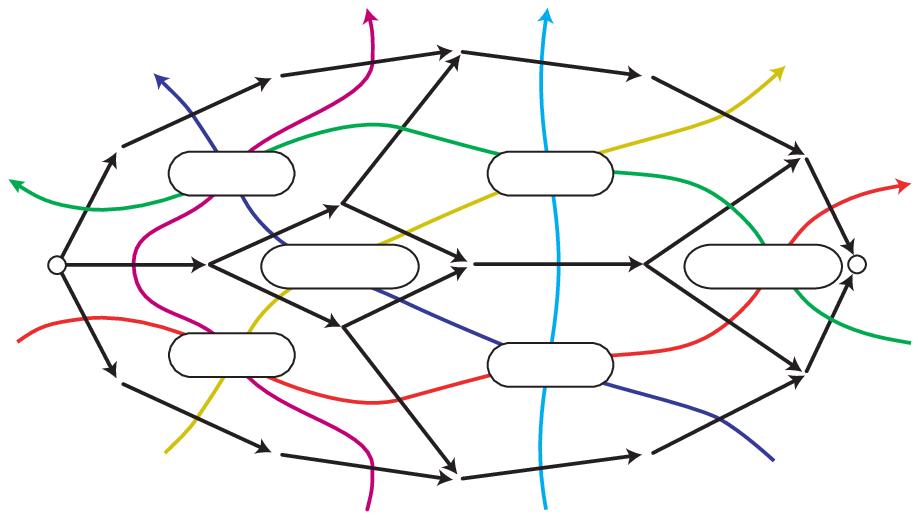}}
\put(24.8,24.8){$\Qs1324$}
\put(67.8,24.8){$\Qs1234$}
\put(-6,15){$\SEP12$}
\put(9,4){$\SEP13$}
\put(32,-2){$\SEP23$}
\put(50,-2){$\SEP14$}
\put(78,3){$\SEP24$}
\put(92,16){$\SEP34$}
\put(17,15.7){$\Ts123$}
\put(17,34){$\Ts234$}
\put(49.7,34){$\Ts134$}
\put(49.7,14.7){$\Ts124$}
\put(2,29){$\ss3$}
\put(2,21){$\ss1$}
\put(12,41){$\ss2$}
\put(12,9){$\ss2$}
\put(30,47){$\ss3$}
\put(30,3){$\ss1$}
\put(8,27){$\ss2$}
\put(20.5,29){$\ss3$}
\put(20.5,21){$\ss1$}
\put(38,36){$\ss2$}
\put(38,14){$\ss2$}
\put(38,30){$\ss1$}
\put(38,19){$\ss3$}
\put(50,27){$\ss2$}
\put(49,47.5){$\ss1$}
\put(49,2){$\ss3$}
\put(65,30){$\ss1$}
\put(65,19){$\ss3$}
\put(68,43){$\ss2$}
\put(68,6){$\ss2$}
\put(82,33){$\ss3$}
\put(82,17){$\ss1$}
\end{picture}
\begin{picture}(88,54)(0,1)
\put(-1,0){\includegraphics{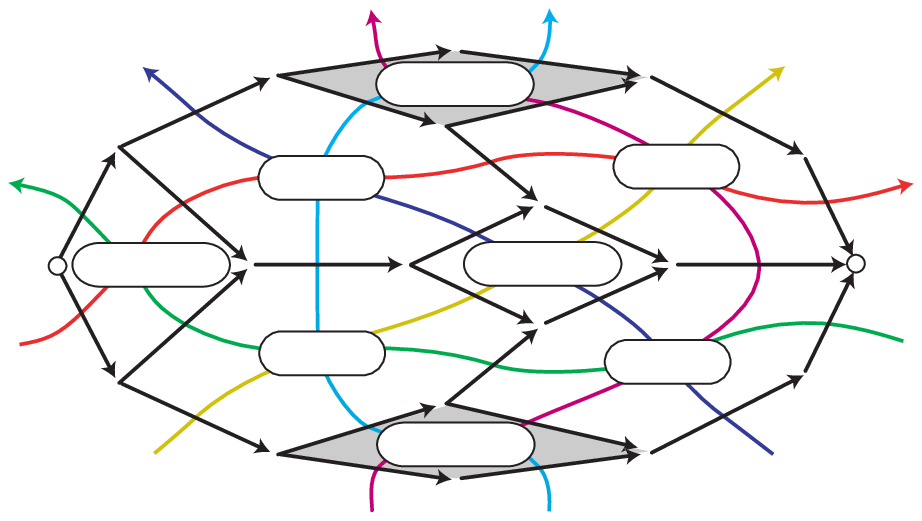}}
\put(36.5,6){$\Qs1423$}
\put(36.5,42.6){$\Qs1423$}
\put(5.3,24.3){$\Qs1234$}
\put(45,24.3){$\Qs1324$}
\put(2,29){$\ss3$}
\put(2,20){$\ss1$}
\put(12,41){$\ss2$}
\put(12,8){$\ss2$}
\put(30,46.5){$\ss3$}
\put(30,40){$\ss1$}
\put(30,9){$\ss3$}
\put(30,2.5){$\ss1$}
\put(15,33){$\ss1$}
\put(15,15){$\ss3$}
\put(50,34){$\ss2$}
\put(50,38.5){$\ss3$}
\put(50,15){$\ss2$}
\put(48,11){$\ss1$}
\put(40,28){$\ss1$}
\put(40,21){$\ss3$}
\put(26,26.5){$\ss2$}
\put(70,26.5){$\ss2$}
\put(49,47){$\ss1$}
\put(49,2){$\ss3$}
\put(56,30){$\ss3$}
\put(56,18.5){$\ss1$}
\put(68,43){$\ss2$}
\put(68,6){$\ss2$}
\put(81,34){$\ss3$}
\put(81.5,16){$\ss1$}
\put(911,33){$\SEP12$}
\put(76,46.5){$\SEP13$}
\put(8,46.5){$\SEP24$}
\put(-9,33){$\SEP34$}
\put(61.7,14.5){$\Ts234$}
\put(26.3,33.3){$\Ts124$}
\put(26.3,15.3){$\Ts134$}
\put(62.5,34.2){$\Ts123$}
\end{picture}
\caption{\sf Two van Kampen diagrams for the pair $(\ss1\ss2\ss1\ss3\ss2\ss1,$
$\ss3\ss2\ss3\ss1\ss2\ss3)$: a tiling by
squares and hexagons witnessing the equivalence of the two expressions. Indexed by pairs,
the separatrices are drawn in grey, inducing a name for every tile. In the top diagram, no two
tiles share the same name, hence the diagram is optimal, and the combinatorial distance
is~$6$; by contrast, two tiles in the bottom diagrams share the name~$\Q1423$, hence the
diagram cannot be optimal.}
\label{F:Van}
\end{figure}

Then the optimality criterion is

\begin{prop}
\label{P:Minimal}
A van Kampen diagram in which any two tiles have different names is optimal.
\end{prop}

\begin{proof}
The proof of Proposition~\ref{P:Lower} shows that, if $\KKK$ contains exactly one
tile named $\T\pp\qq\rr$, then the order of  $\P\pp\qq$, $\P\pp\rr$, $\P\qq\rr$ has
changed between~$\SS(\uu)$ and~$\SS(\vv)$. Hence, under the hypothesis, if
$\KKK$ contains $N$~hexagons with pairwise distinct names, then we have
$\II_3(\SS(\uu), \SS(\vv)) \ge N$. Similarly, if $\KKK$ contains exactly one
tile named $\Q\pp\qq\ppp\qqq$, then the order of  $\P\pp\qq$ and $\P\ppp\qqq$ has
changed between~$\SS(\uu)$ and~$\SS(\vv)$. So, if
$\KKK$ contains $N'$~squares with pairwise distinct names, then we have
$\II_{2,2}(\SS(\uu), \SS(\vv)) \ge N'$. Applying Proposition~\ref{P:Lower}, we deduce
$\dist(\uu, \vv) \ge \NN + \NN'$, hence $\dist(\uu, \vv) = \NN + \NN'$.
\end{proof}

\subsection{Separatrices}
\label{S:Sep}

Having given names to the edges of a van Kampen diagram, we can now draw curves that
connect the edges with the same name.

\begin{defi}
Assume that $\KKK$ is a van Kampen diagram for~$(\uu, \vv)$, and $\P\pp\qq$ is an
inversion of the permutation~$\pi$ represented by~$\uu$ and~$\vv$, \ie, we have $(\qq -
\pp)(\pi(\qq) - \pi(\pp)) < 0$. The
\emph{$\P\pp\qq$-separatrix} in~$\KKK$ is the curve~$\SEP\pp\qq$ obtained by 
connecting the middle points of the (diametrally opposed) edges named~$\P\pp\qq$ inside
each tile, oriented so that the edges of~$\uu$ are crossed first, and those of~$\vv$ last. 
\end{defi}

If $\P\pp\qq$ is not an inversion of the involved permutation, we may consider that
$\SEP\pp\qq$ still exists, but it lies outside the diagram and cuts no edge.

So, by definition, separatrices are obtained by connecting patterns of the form
$$\begin{picture}(98,22)(0,0)
\put(4,0){\includegraphics{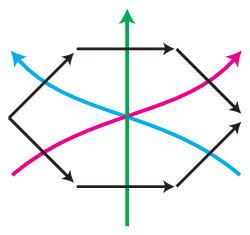}}
\put(0,3){$\SEP\pp\qq$}
\put(17,0){$\SEP\pp\rr$}
\put(28,3){$\SEP\qq\rr$}
\put(33,10){for type~I, \ and}
\put(62,4){\includegraphics{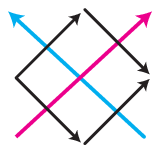}}
\put(58,3){$\SEP\pp\qq$}
\put(77,3){$\SEP\ppp\qqq$}
\put(82,10){for type II.}
\end{picture}$$
To make~$\SEP\pp\qq$ unique, we might require in addition that the curve we
choose inside each tile is the image of the median lines under the affine transformation that
maps a regular polygon to the considered tile (provided the latter is convex). 

\goodbreak
Then we have the following optimality criterion involving separatrices.

\begin{prop}
\label{P:MinimalSep}
A van Kampen diagram in which  any two separatrices cross at
most once is optimal.
\end{prop}

\begin{proof}
By construction, for each triple~$\T\pp\qq\rr$ in~$\Int\nn^{(3)}$, the only place
where $\SEP\pp\qq$ and $\SEP\pp\rr$ may intersect is a type~I tile
named~$\T\pp\qq\rr$ and, conversely, the three separatrices appearing in a type~I tile
pairwise cross one another. Similarly,  for each~$\Q\pp\qq\ppp\qqq$ in~$\Int\nn^{(2,2)}$,
the only place where $\SEP\pp\qq$ and $\SEP\ppp\qqq$ may intersect is a
type~II tile named~$\Q\pp\qq\ppp\qqq$, and, conversely, the two separatrices do cross in
such a type~II tile. Summarizing, the number of times two separatrices $\SEP\pp\qq$
and $\SEP\ppp\qqq$ cross is exactly the number of type~I tiles named
$\T\pp\qq{\ppp,\qqq}$ if the set $\{\pp,\qq,\ppp,\qqq\}$ has three elements, and the
number of type~II tiles named
$\Q\pp\qq\ppp\qqq$ if it has four elements. Hence all tiles in~$\KKK$ have pairwise
different names if and only if any two separatrices cross at most once in~$\KKK$, and we
apply Proposition~\ref{P:Minimal}.
\end{proof}

\section{Subword reversing}
\label{S:Reversing}

Up to now, we considered the combinatorial distance between two expressions~$\uu, \vv$
representing some permutation, \ie, the minimal number of braid relations needed to
transform~$\uu$ into~$\vv$. We now address a related, but different question, namely
the number of braid relations needed to transform~$\uu$ into~$\vv$ when one uses the
specific strategy called \emph{subword reversing}~\cite{Dff, Dgc}. The latter is known to
solve the word problem, \ie, to provide a step-by-step transformation of~$\uu$ into~$\vv$
by means of braid relations when this is possible, \ie, when $\uu$ and~$\vv$ are
equivalent reduced expressions.

Investigating the optimality of the reversing method, \ie, comparing the number of braid
relations used by subword reversing and the combinatorial distance, is a natural question.
Little is known at the moment. Here, we shall establish several partial results. We observe in
particular that the reversing method need not be optimal but, using the geometric criteria of
Sections~\ref{S:Van} and~\ref{S:Sep}, we characterize cases when reversing turns out to be
optimal. Also, in a slightly different context of non-equivalent initial expressions, we shall
prove that there exist length~$\ell$ expressions~$\uu, \vv$ whose reversing require
using~$O(\ell^4)$ braid relations, which is not redundant with Proposition~1.

\subsection{Subword reversing and reversing diagrams}

We recall that an $\nn$-expres\-sion is a word on the alphabet $\{\ss1, ..., \ss\nno\}$.
Hereafter, we introduce a second alphabet $\{\sss1, ..., \sss\nno\}$ which is a formal copy of
the previous one. A word over the extended alphabet $\{\ss1, ..., \ss\nno, \sss1, ...,
\sss\nno\}$ will be called an \emph{extended $\nn$-expression}. For each extended
expression~$\uu$, we denote by~$\sym\uu$ the extended expression obtained from~$\uu$
by reversing the order of the letters and exchanging~$\ss\ii$ and~$\sss\ii$ everywhere. So,
for instance, we have $\sym{\ss1\ss2} = \sss2\; \sss1$.

We now introduce a binary relation~$\rev$ (or rewrite rule) on extended expressions.

\begin{defi}
If $\ww, \ww'$ are extended expressions, we declare that $\ww \rev \ww'$
holds if $\ww'$ is obtained from~$\ww$ by replacing some subword~$\sss\ii \ss\jj$
either by $\ss\jj \ss\ii \sss\jj \sss\ii$ (if $\vert \ii - \jj\vert = 1$ holds), or by $\ss\jj
\sss\ii$ (if $\vert\ii - \jj\vert \ge 2$ holds), or by the empty word~$\ew$ (if $\jj = \ii$
holds). A finite sequence $(\ww_0, ..., \ww_\NN)$ is called a \emph{reversing
sequence} if $\ww_\kk \rev \ww_{\kk+1}$ holds for each~$\kk$. Finally, we
say that $\ww$ \emph{reverses} to~$\ww'$ if there exists a reversing
sequence $(\ww_0, ..., \ww_\NN)$ satisfying $\ww_0 = \ww$ and $\ww_\NN = \ww'$.
\end{defi}

The principle of reversing is to push the letters~$\sss\ii$  to the right, and the
letters~$\ss\jj$ to the left, until no subword~$\sss\ii \ss\jj$ remains.  For
instance, $\sss1\sss2\sss1\ss2\ss1\ss2$ reverses to the empty word, as we have

\medskip\noindent 
$(*)$ \  $\sss1\sss2\BOLD{\sss1\ss2}\ss1\ss2
\rev \sss1\BOLD{\sss2\ss2}\ss1\sss2\sss1\ss1\ss2
\rev \BOLD{\sss1\ss1}\sss2\sss1\ss1\ss2
\rev \sss2\BOLD{\sss1\ss1}\ss2
\rev \BOLD{\sss2\ss2}
\rev \ew$,

\medskip\noindent 
where, at each step, the subword that will be reversed has been written
in bold.

As will become clear below, a reversing sequence from~$\sym\uu \vv$ to the empty word
provides a distinguished way of transforming~$\uu$ into~$\vv$ by means of braid relations
and, therefore, there exists an associated van Kampen diagram. Now, it follows from the
particular definition of subword reversing that the associated van Kampen diagrams have
specific properties, namely they can essentially be drawn on a rectangular grid, a specific
point that will be important in the sequel. 

These diagrams that are essentially van Kampen diagrams will be called \emph{reversing
diagrams}.  As a van Kampen diagram, a reversing diagram consists of edges labeled by
letters~$\ss\ii$. The specific point is that, in a reversing diagram, all edges are either
horizontal right-oriented edges and vertical down-oriented edges, and that, in addition to
the latter, there may exist
$\ew$-labeled arcs. Assume that $(\ww_0, ..., \ww_\NN)$ is a reversing sequence, hence a
sequence of extended expressions containing both types of letters~$\ss\ii$ and~$\sss\jj$.
First, we draw a connected path indexed by the successive letters of~$\ww_0$ by attaching a
horizontal arrow $\stackrel{\ss\ii}{\rightarrow}$ with each letter~$\ss\ii$, and a vertical
arrow~$\downarrow^{\!\ss\ii}$ with each letter~$\sss\ii$. Then, we inductively complete
the diagram as follows. Assume that one goes from~$\ww_{\kk-1}$ to~$\ww_\kk$ by
reversing some subword~$\sss\ii\ss\jj$. By induction hypothesis, the latter
subword~$\sss\ii\ss\jj$ corresponds to an open pattern 
\raisebox{-3mm}{
\begin{picture}(9,8)
\put(2,1){\includegraphics[scale=0.5]{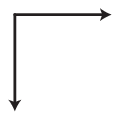}}
\put(0,3){$\scriptstyle\ss\ii$}
\put(3.5,6.5){$\scriptstyle\ss\jj$}
\end{picture}}
in the diagram. Then we complete that pattern with new
arrows, according to the rule
$$\begin{picture}(121,13)(0,0)
\put(0,0){\includegraphics{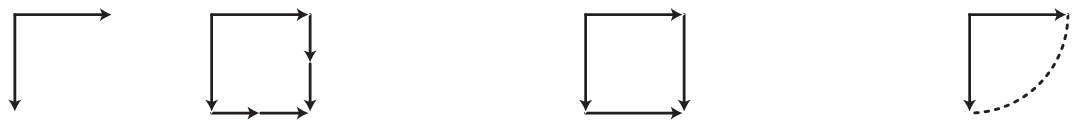}}
\put(-2,5){$\ss\ii$}
\put(5,12){$\scriptstyle\ss\jj$}
\put(10,4){$\to$}

\put(17,5){$\ss\ii$}
\put(24,12){$\ss\jj$}
\put(22,-1.5){$\ss\jj$}
\put(26,-1.5){$\ss\ii$}
\put(31.5,8){$\ss\ii$}
\put(31.5,3){$\ss\jj$}
\put(35,7){for $\vert\ii\!-\!\jj\vert{=}1$,}
\put(37,3){(type I)}

\put(55.5,5){$\ss\ii$}
\put(63,12){$\ss\jj$}
\put(69.5,5){$\ss\ii$}
\put(63,-1.5){$\ss\jj$}
\put(73,7){for $\vert\ii\!-\!\jj\vert{\ge}2$,}
\put(75,3){(type II)}

\put(94.5,5){$\ss\ii$}
\put(101,12){$\ss\jj$}
\put(105,2){$\ew$}
\put(109,7){for $\ii{=}\jj$,}
\put(107.5,3){(type III)}
\end{picture}$$
with the convention that $\ew$-labeled dotted arcs, dotted arcs in the diagrams, are
subsequently ignored. For instance, the reversing diagram associated with the above
reversing sequence~$(*)$ is displayed in Figure~\ref{F:Reversing}.

\begin{rema}
Because of the $\ew$-labelled arcs, the above patterns are not the most general ones
appearing in a reversing diagram. The general patterns are actually
$$\begin{picture}(121,13)(0,0)
\put(0,0){\includegraphics{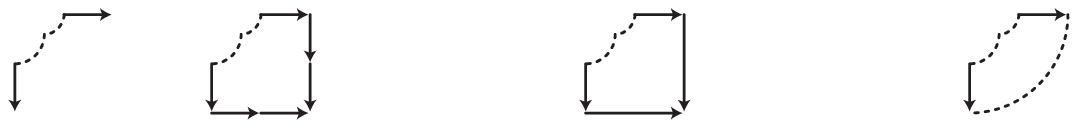}}
\put(2,9){$\ew$}
\put(-2.5,3){$\ss\ii$}
\put(6.5,12){$\ss\jj$}
\put(10,4){$\to$}

\put(22,9){$\ew$}
\put(17,3){$\ss\ii$}
\put(26.5,12){$\ss\jj$}
\put(22,-1.5){$\ss\jj$}
\put(26,-1.5){$\ss\ii$}
\put(31.5,8){$\ss\ii$}
\put(31.5,3){$\ss\jj$}
\put(35,7){for $\vert\ii\!-\!\jj\vert{=}1$,}
\put(37,3){(type I)}

\put(60,9){$\ew$}
\put(55.5,3){$\ss\ii$}
\put(65.5,12){$\ss\jj$}
\put(69.5,5){$\ss\ii$}
\put(63,-1.5){$\ss\jj$}
\put(73,7){for $\vert\ii\!-\!\jj\vert{\ge}2$,}
\put(75,3){(type II)}

\put(99,9){$\ew$}
\put(94.5,3){$\ss\ii$}
\put(103.5,12){$\ss\jj$}
\put(105,2){$\ew$}
\put(109,7){for $\ii{=}\jj$.}
\put(107.5,3){(type III)}
\end{picture}$$
\end{rema}

\begin{figure}[tb]
\begin{picture}(35,21)(0,1)
\put(0,0){\includegraphics{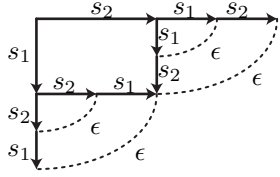}}
\put(-3,2){$\ss1$}
\put(-3,7){$\ss2$}
\put(-3,15){$\ss1$}
\put(8,21){$\ss2$}
\put(19,21){$\ss1$}
\put(26,21){$\ss2$}
\put(17,17.5){$\ss1$}
\put(17,12.5){$\ss2$}
\put(3,11){$\ss2$}
\put(11,11){$\ss1$}
\put(14,1.5){$\ew$}
\put(8,5){$\ew$}
\put(24,15){$\ew$}
\put(29.5,11.5){$\ew$}
\end{picture}
\caption{\sf Reversing diagram associated with the sequence~$(*)$, \ie, the reversing
diagram of the pair $(\ss1\ss2\ss1, \ss2\ss1\ss2)$. Metric aspects are ignored, so the
reversing diagram is always considered up to a piecewise affine deformation.}
\label{F:Reversing}
\end{figure}

In this way, we associate with every reversing sequence a reversing diagram. Conversely, it is easy to see that, starting with a diagram as above, we can recover a (not
necessarily unique) reversing sequence by reading the labels of the various paths going
from the bottom-left corner to the top-right corner, and using the convention that a vertical
$\ss\ii$-labeled edge contributes~$\sss\ii$. 

By construction, subword reversing may be applied to any initial word consisting of
letters~$\ss\ii$ and~$\sss\ii$, and not only to words of the form~$\sym\uu \vv$ where
$\uu$ and~$\vv$ are equivalent reduced expressions. So there is no ambiguity in the
following notion.

\begin{defi}
If $\uu, \vv$ a	are expressions, the \emph{reversing diagram for~$(\uu, \vv)$} is the
reversing diagram starting with the word~$\sym\uu \vv$. 
\end{defi}

By construction, the reversing diagram for the pair~$(\uu, \vv)$ starts with a vertical
down-oriented path labeled~$\uu$ and a horizontal right-oriented path labeled~$\vv$
starting from a common source. For instance, the diagram of Figure~\ref{F:Reversing} is the
reversing diagram for the pair~$(\ss1\ss2\ss1, \ss2\ss1\ss2)$.

By construction, a reversing diagram becomes a van Kampen diagram when all $\ew$-arcs
are collapsed---but we do \emph{not} do it, and keep the diagrams as they are now,
insisting that they are drawn in a rectangular grid. A reversing diagram contains three types
of tiles:

- type~I tiles, which are hexagons, and correspond to type~I braid
relations, 

- type~II tiles, which are squares, and correspond to type~II braid relations,

- type~III tiles, which are digons, and correspond to free group relations $\sss\ii \ss\ii =
1$. The latter will be called \emph{trivial}.

The connection between subword reversing and the problem of recognizing equivalent
reduced expressions of a permutation is given by the following result.

\begin{prop} \cite{Dff}
\label{P:Complete}
$(i)$ For all reduced expressions~$\uu, \vv$, there exists a unique pair of reduced
expressions~$\uu', \vv'$ such that $\sym\uu \vv$ reverses to~$\vv'\sym{\uu'}$.

$(ii)$ Two reduced expressions $\uu, \vv$ represent the same permutation if and only if\/
$\sym\uu \vv$ reverses to the empty word.
\end{prop} 

In the case of~$(ii)$, one implication is clear: by construction, a reversing sequence
from~$\sym\uu
\vv$ to the empty word gives a reversing diagram for~$(\uu, \vv)$ that concludes with
$\ew$-arcs everywhere on the bottom and the right, hence, after collapsing the
$\ew$-edges, it gives a van Kampen diagram for~$(\uu, \vv)$, thus proving that $\uu$
and~$\vv$ are equivalent. The converse implication is not obvious, as not every van
Kampen diagram comes from a reversing diagram. The specific point is that, in a reversing
diagram, two edges at most start from any vertex, a property that fails in the top diagram of
Figure~\ref{F:Van}: so that diagram is certainly not associated with a reversing. The proof of
Proposition~\ref{P:Complete}---which actually extends to arbitrary braids---relies on the
so-called Garside theory of braids~\cite{Gar, Eps}.

In terms of reversing diagrams,  the situation is as follows. In the particular case when $\uu$
and~$\vv$ are equivalent reduced expressions, the reversing diagram for $(\uu, \vv)$
finishes for $\ew$-labeled arcs everywhere, and collapsing all these $\ew$-labeled arcs yields
a van Kampen diagram for~$(\uu, \vv)$. In general, if $\uu, \vv$ are arbitrary expressions,
then the reserving diagram for $(\uu, \vv)$ is still finite, and it finishes with arrows forming
a word of the form~$\vv' \sym{\uu'}$ where $\uu', \vv'$ are two expressions that need not
be empty. Then collapsing all the $\ew$-labeled arcs yields a van Kampen diagram for the
pair~$(\uu\vv', \vv\uu')$. It can be shown that, if $\uu$ and $\vv$ are reduced, then
$\uu\vv'$ and~$\vv\uu'$ are reduced and equivalent---the
latter point is obvious as, by construction, the reversing
diagram for~$(\uu,\vv)$ provides a van Kampen diagram for
$(\uu\vv', \vv\uu')$---and the permutation represented by~$\uu\vv'$ and~$\vv\uu'$
is the least upper bound of the permutations represented by~$\uu$ and~$\vv$ with respect
to the weak order of~$\Sym_\nn$~\cite{BjB}.

In every case, the number of nontrivial tiles in the reversing diagram or, equivalently, the
number of nontrivial steps in an associated reversing sequence, is well defined.

\begin{defi}
Assume that $\uu, \vv$ are reduced expressions. The \emph{reversing
complexity} of~$(\uu, \vv)$, denoted~$\compl(\uu,\vv)$, is the number of
nontrivial tiles in the reversing diagram for~$(\uu, \vv)$.
\end{defi}

Equivalently, the reversing complexity~$\compl(\uu,\vv)$ is the number of nontrivial
steps in a  reversing sequence from~$\sym\uu \vv$ to a word of the form~$\vv'
\sym{\uu'}$. By the above remarks, a reversing diagram for~$(\uu,\vv)$ with
$\NN$~nontrivial tiles provides a van Kampen diagram for~$(\uu\vv', \vv\uu')$ with
$\NN$~tiles, where $\vv' \sym{\uu'}$ is the final word of the reversing process. So we
always have
\begin{equation}
\label{E:Ineq0}
\dist(\uu\vv',\vv\uu') \le \compl(\uu,\vv).
\end{equation}
In particular, when we start with equivalent reduced expressions~$\uu, \vv$, we have
\begin{equation}
\label{E:Ineq}
\dist(\uu,\vv) \le \compl(\uu,\vv),
\end{equation}
since, in this case, the final expressions~$\uu'$ and~$\vv'$ are empty by
Proposition~\ref{P:Complete}$(ii)$.

We shall now discuss the converse inequality, \ie, the question of
whether subword reversing, viewed as a particular strategy for finding derivations between
equivalent expressions of a permutation, is efficient, or even possibly optimal.

\begin{prop}
\label{P:NotOptimal}
Subword reversing is not always optimal: There exist equivalent reduced expressions~$\uu,
\vv$ satisfying $\dist(\uu,\vv) < \compl(\uu,\vv)$.
\end{prop}

\begin{proof}
Consider the $4$-expressions~$\uu = \ss1 \ss2 \ss1 \ss3 \ss2 \ss1$ and $\vv =
\ss3 \ss2 \ss3 \ss1 \ss2 \ss3$, two expressions of the flip permutation~$\phi_4$
of~$\Sym_4$. Together with Proposition~\ref{P:Minimal}, the top
diagram in~Figure~\ref{F:Van} gives $\dist(\uu,\vv) = 6$. On the other hand, the reversing
diagram of Figure~\ref{F:Dist} gives $\compl(\uu,\vv) = 8$.
\end{proof}

\begin{figure}[htb]
\begin{picture}(80,54)(0,0)
\put(0,0){\includegraphics{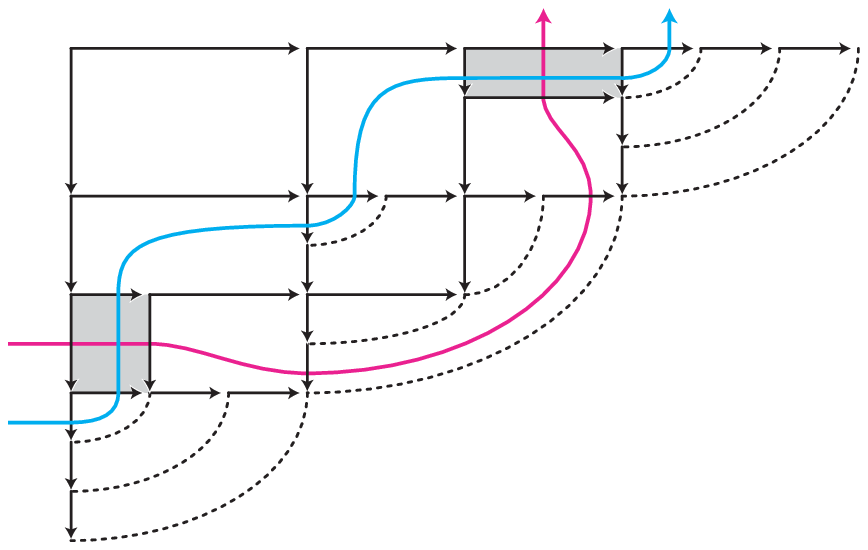}}
\put(2,3){$\ss1$}
\put(2,8){$\ss2$}
\put(2,13){$\ss3$}
\put(2,21){$\ss1$}
\put(15,21){$\ss1$}
\put(31,22){$\ss1$}
\put(31,18){$\ss2$}
\put(2,30){$\ss2$}
\put(2,43){$\ss1$}
\put(31,43){$\ss1$}
\put(31,32){$\ss2$}
\put(31,27.5){$\ss3$}
\put(17,52){$\ss3$}
\put(17,36.5){$\ss3$}
\put(32,36.5){$\ss2$}
\put(40,36.5){$\ss1$}
\put(48,36.5){$\ss3$}
\put(55.5,36.5){$\ss2$}
\put(21,26.5){$\ss2$}
\put(8,26.5){$\ss3$}
\put(37,26.5){$\ss1$}
\put(8,16.5){$\ss3$}
\put(16,16.5){$\ss2$}
\put(24,16.5){$\ss1$}
\put(37,52){$\ss2$}
\put(50,52){$\ss3$}
\put(64,52){$\ss1$}
\put(72,52){$\ss2$}
\put(80,52){$\ss1$}
\put(46.5,48){$\ss1$}
\put(63,48){$\ss1$}
\put(46.5,40){$\ss2$}
\put(46.5,30){$\ss3$}
\put(46.5,40){$\ss2$}
\put(63,42.5){$\ss2$}
\put(63,37.5){$\ss1$}
\put(-10,20){$\SEP23$}
\put(-10,12){$\SEP14$}
\end{picture}
\caption{\sf Reversing diagram for
$\sss1\sss2\sss3\sss1\sss2\sss1\ss3\ss2\ss3\ss1\ss2\ss1$. There are eight nontrivial tiles
(four type~I hexagons and four type~II squares), giving $\compl(\ss1\ss2\ss1\ss3\ss2\ss1,
\ss3\ss2\ss3\ss1\ss2\ss1) = 8$. Collapsing the $\ew$-edges in the above diagram yields
the bottom van Kampen diagram of Figure~\ref{F:Van}. The failure of optimality is
witnessed by the two intersections of the separatrices $\SEP14$ and
$\SEP23$.}
\label{F:Dist}
\end{figure}

\subsection{An optimality criterion}

Despite the negative result of Proposition~\ref{P:NotOptimal}, experiments show that
subword reversing is often an efficient strategy. What we do now is to establish sufficient
criteria for recognizing that reversing is possibly optimal. Of course, we shall say that a
reversing diagram~$\DDD$ is \emph{optimal} if the van Kampen diagram obtained by
collapsing the $\ew$-labeled arcs in~$\DDD$ is optimal, \ie, if it realizes the combinatorial
distance between the boundary expressions.

\begin{prop}
\label{P:Optimal}
A reversing diagram containing no digon, \ie, containing only tiles of type~I and~II, is
optimal. 
\end{prop}

\begin{proof}
Assume that $\DDD$ is the reversing diagram for~$(\uu, \vv)$.
By definition, the separatrices of~$\DDD$ start from the edges corresponding
to~$\uu$, \ie, here, from the left. Then an induction on the number of tiles shows that only
the following orientations may appear in the tiles of~$\DDD$.
$$\begin{picture}(74,20)(0,-2)
\put(0,0){\includegraphics{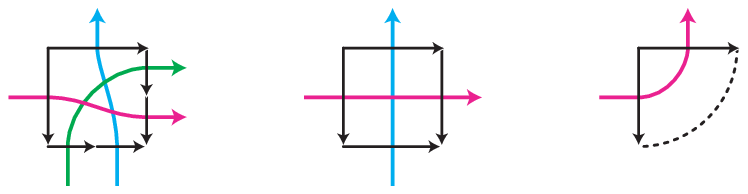}}
\put(-4,8){$\Sep$}
\put(26,8){$\Sep$}
\put(56,8){$\Sep$}
\put(4,-3){$\Sep''$}
\put(10,-3){$\Sep'$}
\put(38,-3){$\Sep'$}
\end{picture}$$
Hence every vertical edge of~$\DDD$ is crossed by a separatrix from left to right, and every
horizontal edge is crossed from bottom to top. Moreover, we see that digons are the only
tiles that can change the orientation of a separatrix from horizontal to vertical---whereas
only hexagons can change the orientation from vertical to horizontal. Also we see that, if
two separatrices cross in a tile, then, when entering that tile, at least one of them is vertical.

Now assume that two separatrices $\Sep, \Sep'$ cross at least twice in~$\DDD$,
Two cases may occur, according to whether the names of these separatrices involve three
or four integers. Assume first that there exist $\pp, \qq, \rr$ satisfying $\Sep =
\SEP\pp\qq$ and~$\Sep' = \SEP\pp\rr$. Then $\Sep$ and~$\Sep'$ can cross
only in hexagons named~$\T\pp\qq\rr$, in which case, putting $\Sep'' =
\SEP\qq\rr$, the separatrices~$\Sep$ and~$\Sep''$, as well as $\Sep'$ and~$\Sep''$,
cross too. Let~$H_1$ (\resp $H_2$) be the first (\resp second) hexagon where
$\Sep$, $\Sep'$ and~$\Sep''$ cross. Without loss of generality, we may assume that $\Sep$
is horizontal when entering~$H_1$  (hence also when exiting it) and that $\Sep'$ is
vertical when exiting~$H_1$ (hence also when entering it). Then $\Sep'$ remains
above~$\Sep$ (measured from the bottom of the diagram)---with $\Sep''$ lying
in between---until they enter the hexagon~$H_2$. Therefore, the only possibility is that
$\Sep'$ is horizontal when entering~$H_2$ (from the left), whereas $\Sep$ is vertical when
entering~$H_2$ (from the bottom). This is possible only if a digon changes the orientation
of~$\Sep$ from horizontal to vertical between~$H_1$ and~$H_2$.

Assume now that there exist $\pp, \qq, \ppp, \qqq$ satisfying $\Sep =
\SEP\pp\qq$ and~$\Sep' = \SEP\ppp\qqq$. Then $\Sep$ and~$\Sep'$ can cross
only in squares named~$\Q\pp\qq\ppp\qqq$. Let $S_1$ (\resp $S_2$) be the first
(\resp second) square where $\Sep$ and~$\Sep'$ cross. Without loss of generality, we can
assume that $\Sep$ is horizontal after (and before)~$S_1$, and $\Sep'$ is vertical. Then, as
in the case of hexagons, $\Sep'$ remains above~$\Sep$ until they enter the square~$S_2$.
Therefore, $\Sep'$ is horizontal when entering~$S_2$ (from the
left), whereas $\Sep$ is vertical when entering~$S_2$ (from the bottom). Hence a digon
changes the orientation of~$\Sep$ from horizontal to vertical between~$S_1$ and~$S_2$.

So, in any case, two separatrices may cross twice only if there is a digon in~$\DDD$. Then
we apply Proposition~\ref{P:MinimalSep}.
\end{proof}

The previous result can be improved by showing that some digons are harmless and
can be ignored. Indeed, consider a pattern of the form $\sss\ii \ss{\ii+1} \ss\ii$. Then we
have
$\sss\ii \ss{\ii+1} \ss\ii \rev \ss{\ii+1} \ss\ii \sss{\ii+1} \sss\ii \ss\ii \rev \ss{\ii+1} \ss\ii
\sss{\ii+1}$,
corresponding to an hexagon with an appended digon
\vrule width0pt depth7mm
$\begin{picture}(27,11)(-4,4)
\put(0,0){\includegraphics{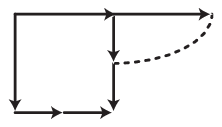}}
\put(-3,5){$\ss\ii$}
\put(11,8){$\ss\ii$}
\put(11,3){$\ss{\ii+1}$}
\put(14,12){$\ss\ii$}
\put(3,12){$\ss{\ii+1}$}
\put(-1,-1.5){$\ss{\ii+1}$}
\put(6,-1.5){$\ss\ii$}
\put(19,5){$\ew$}
\end{picture}$
in the diagram. Let us introduce two new types of hexagonal tiles, namely, for
$\vert\ii-\jj\vert = 1$, 
$$\begin{picture}(72,18)(0,0)
\put(15.5,5){\includegraphics{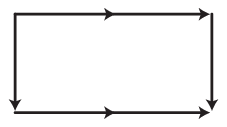}}
\put(0,10){type I': }
\put(13,10){$\ss\ii$}
\put(20,17.5){$\ss\jj$}
\put(30,17.5){$\ss\ii$}
\put(20,3.5){$\ss\jj$}
\put(30,3.5){$\ss\ii$}
\put(37,10){$\ss\jj$,}
\put(45,10){type I'': }
\put(60,0){\includegraphics{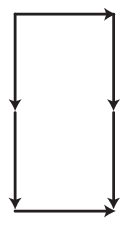}}
\put(57.5,15){$\ss\ii$}
\put(57.5,5){$\ss\jj$}
\put(71.5,15){$\ss\ii$}
\put(71.5,5){$\ss\jj$}
\put(64,22){$\ss\jj$}
\put(64,-2){$\ss\ii$}.
\end{picture}$$
Unsing such tiles amounts to replacing two adjacent tiles with one unique tile of the
new type, but they do not change anything in the rest of the diagram. In this way, we
obtain a new type of reversing diagrams that we call \emph{compacted}. Note that the
compacted diagram associated with an initial expression~$\sym\uu \vv$ need not be unique,
as there may be several ways of grouping the tiles, see Figure~\ref{F:Compact}. However, as
the situation after a tile of type~I' or~I'' is exactly the same as the situation after
the corresponding type~I tile followed by a type~III tile, the number of nontrivial tiles
is the same in any diagram associated with a given initial pair~$(\uu, \vv)$.

\begin{figure}[tb]
\begin{picture}(90,38)(0,0)
\put(2,0){\includegraphics{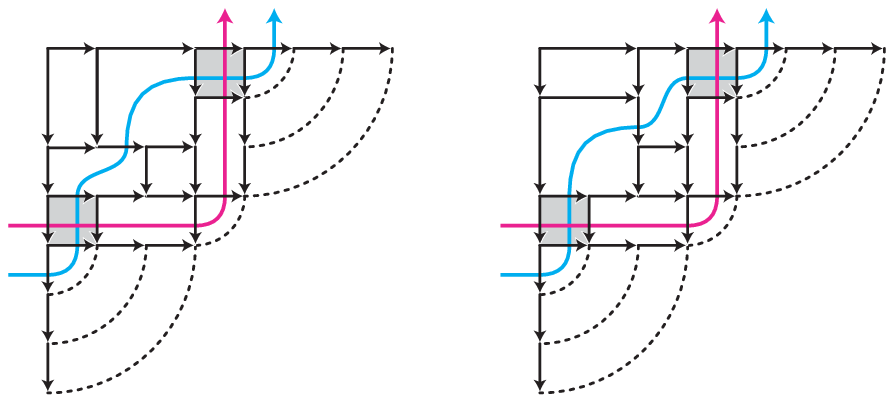}}
\put(2,3){$\ss1$}
\put(2,8){$\ss2$}
\put(2,13){$\ss3$}
\put(2,18){$\ss1$}
\put(2,23){$\ss2$}
\put(2,30){$\ss1$}
\put(6,37){$\ss3$}
\put(14,37){$\ss2$}
\put(20,37){$\ss3$}
\put(25,37){$\ss1$}
\put(31,37){$\ss2$}
\put(36,37){$\ss3$}
\put(-8,17){$\SEP23$}
\put(-8,12){$\SEP14$}
\put(52,3){$\ss1$}
\put(52,8){$\ss2$}
\put(52,13){$\ss3$}
\put(52,18){$\ss1$}
\put(52,23){$\ss2$}
\put(52,30){$\ss1$}
\put(56,37){$\ss3$}
\put(64,37){$\ss2$}
\put(70,37){$\ss3$}
\put(75,37){$\ss1$}
\put(81,37){$\ss2$}
\put(86,37){$\ss3$}
\put(42,17){$\SEP23$}
\put(42,12){$\SEP14$}
\end{picture}
\caption{\sf Two slightly different ways of compacting the reversing diagram of
Figure~\ref{F:Dist}.}
\label{F:Compact}
\end{figure}

\goodbreak
The expected improvement of Proposition~\ref{P:Optimal} is

\begin{prop}
\label{P:OptimalBis}
A compacted reversing diagram containing no digon, \ie, containing only tiles of types~I, I',
I'', and II, is optimal. 
\end{prop}

\begin{proof}
The new tiles of type~I' and~I'' do not change the orientation of separatrices. Indeed, the
corresponding possibilities are
$$\begin{picture}(60,23)(0,2)
\put(0,0){\includegraphics{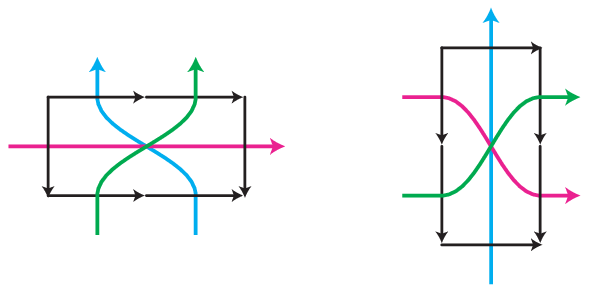}}
\end{picture}$$
None of these tiles changes the horizontal-vertical orientation of the separatrices, and,
therefore, their appearing in a reversing diagram does not affect the argument used in
the proof of Proposition~\ref{P:Optimal}. 
\end{proof}

An application of the above criterion will be mentioned in Remark~\ref{R:Optimal} below.

\subsection{Upper bounds}

Very little is known about the reversing complexity in general. In particular, the following
conjecture, which is the natural counterpart of Proposition~\ref{P:Upper}, remains open
at the moment---see~\cite{Aut} for partial results.

\begin{conj}
\label{C:UpperNL}
For all $\nn$-expressions~$\uu, \vv$ of length~$\ell$, the reversing complexity
\linebreak
$\compl(\uu,\vv)$ lies in~$O(\nn^2\ell)$.
\end{conj}

Even the weaker result of the reversing complexity being  polynomial is not known. By
adapting the method used for~\cite[Prop.~3]{Dhg}, one comes up with the weak result that,
if $\uu, \vv$ are length~$\ell$ expressions, then $\dist(\uu,\vv) \le C \cdot 81^\ell$
holds for some constant~$C$. Using a careful analysis of separatrices, one can obtain the
following improvement.

\begin{prop}
\label{P:UpperRev}
If $\uu, \vv$ are expressions of length~$\ell$, then we have $\dist(\uu,\vv) \le C \cdot
9^\ell$ for some constant~$C$.
\end{prop}

As the argument is complicated and, at the same time, the result seems far from optimal, we
skip the proof and refer to~\cite{Aut} for details.

\begin{rema*}
The index~$\nn$ does not appear in Proposition~\ref{P:UpperRev}. This
reflects the fact that, although the maximal reversing complexity between two
$n$-expressions of length~$\ell$ increases with~$n$ and~$\ell$, it does not increase
indefinitely: if we denote by~$\NN(n,\ell)$ the maximal reversing complexity between two $n$-expressions of length~$\ell$, then, for each~$\ell$, the value
of~$\NN(n,\ell)$ is constant for $n\ge 2\ell$. This is due to the fact that, when
$\nn$ is too large with respect to~$\ell$, the indices of the transpositions~$\ss\ii$
occurring in an $\nn$-expression cannot cover the whole of~$\Int\nn$ and commutation
relations of type~\ref{E:Braid2} occur. Here again, we refer to~\cite{Aut} for more details.
\end{rema*}

\subsection{A lower bound for the reversing complexity}

The $\nn$-expressions used in the proof of Proposition~\ref{P:Lower} to establish the
inequality~\eqref{E:Example} have length $\ell = \nn(\nn-1)/2$ so that, in this way, we
obtain for infinitely many values of~$\ell$ equivalent reduced expressions of length~$\ell$
that satisfy $\dist(\uu, \vv) \ge \ell(\ell-1)/2$, hence, a fortiori,
\begin{equation}
\label{E:LowerBis}
\compl(\uu, \vv) \ge \frac{\ell(\ell-1)}2.
\end{equation}

\begin{ques}
\label{Q:LowerRev}
Can one construct a sequence $(\uu_\ell, \vv_\ell)$ of pairwise equivalent reduced
expressions of length~$\ell$ such that $\compl(\uu_\ell, \vv_\ell)$ is more than quadratic
in~$\ell$?
\end{ques}

We leave Question~\ref{Q:LowerRev} open, but we now
address another related question and establish a result that
illustrates how complicated the reversing process may be.

\begin{prop}
\label{P:Quartic}
For each~$\ell$ there exist expressions $\uu, \vv$ of length~$\ell$ satisfying
$$\compl(\uu, \vv) \ge \frac43 \ell^4$$
for $\ell$ large enough.
\end{prop}

We begin with an auxiliary lemma. Hereafter  we write $w \revk{k} w'$ if there is a
length~$\kk$ reversing sequence from~$w$ to~$w'$, \emph{not} counting trivial steps of
type~III.

\begin{lemm}
\label{L:Quartic1}
For $\ii, \pp \ge 1$, put 
$\ea_{i,p} = \ss{i+p-1}\ss{i+p-2} ... \ss{i}$, $\eb_{i,p} = \ss\ii\ss{i+1} ... 
\ss{i+p-1}$, $\ec_{\ii,\pp} = \ea_{i,p} \ea_{i+1,p}$, and 
$\ed_{\ii,\pp} = \eb_{i+1,p} \eb_{i,p}$. Then, for all $i,
p$, we have 
\begin{gather}
\label{E:BA}
\sym{\eb_{i,p}} \; \ea_{i+1,p}  \revk{\Nba_\pp}
\ea_{i,p+1} \; \sym{\eb_{i,p+1}},\\
\label{E:DC}
\sym{\ed_{i,p}}  \; \ec_{i+2,p} 
\revk{\Ndc_\pp}
\ec_{i,p+2}  \; \sym{\ed_{i,p+2}},
\end{gather}
with $\Nba_\pp =p^2+p-1$ and $\Ndc_\pp =4p^2+8p-3$.
\end{lemm}

\begin{figure}[htb]
\begin{picture}(80,72)(0,4)
\put(0,0){\includegraphics{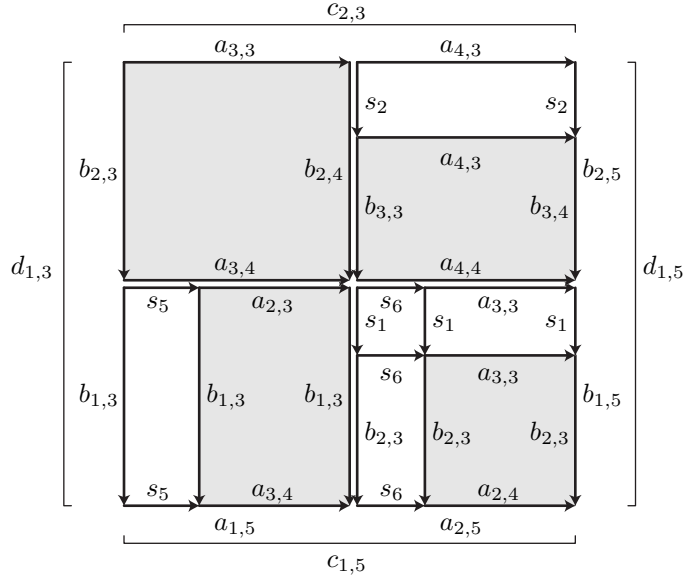}}
\put(4,54){$\eb_{2,3}$}
\put(34,54){$\eb_{2,4}$}
\put(4,24){$\eb_{1,3}$}
\put(71,54){$\eb_{2,5}$}
\put(71,24){$\eb_{1,5}$}
\put(-5,41){$\ed_{1,3}$}
\put(79,41){$\ed_{1,5}$}
\put(22,7){$\ea_{1,5}$}
\put(52,7){$\ea_{2,5}$}
\put(37,2){$\ec_{1,5}$}
\put(22,70.5){$\ea_{3,3}$}
\put(52,70.5){$\ea_{4,3}$}
\put(22,41.5){$\ea_{3,4}$}
\put(52,55.5){$\ea_{4,3}$}
\put(52,41.5){$\ea_{4,4}$}
\put(37,75.5){$\ec_{2,3}$}
\put(42,49){$\eb_{3,3}$}
\put(64,49){$\eb_{3,4}$}
\put(42,63){$\ss2$}
\put(66,63){$\ss2$}
\put(13,36.5){$\ss5$}
\put(27,36.5){$\ea_{2,3}$}
\put(13,11.5){$\ss5$}
\put(21,24){$\eb_{1,3}$}
\put(34,24){$\eb_{1,3}$}
\put(27,11.5){$\ea_{3,4}$}
\put(42,34){$\ss1$}
\put(51,34){$\ss1$}
\put(66,34){$\ss1$}
\put(42,19){$\eb_{2,3}$}
\put(51,19){$\eb_{2,3}$}
\put(64,19){$\eb_{2,3}$}
\put(44,36.5){$\ss6$}
\put(44,27){$\ss6$}
\put(44,11.5){$\ss6$}
\put(57,36.5){$\ea_{3,3}$}
\put(57,27){$\ea_{3,3}$}
\put(57,11.5){$\ea_{2,4}$}
\end{picture}
\caption{\sf Proof of Relation~\ref{E:DC}, here with $\ii = 1$ and $\pp = 3$; the grey
rectangles correspond to Relation~\ref{E:BA}.}
\label{F:Quadratic}
\end{figure}

\begin{proof}
For~\eqref{E:BA} we use induction on~$p$. The
case $p=1$ is
$$\sym \eb_{i,1}  \; \eb_{i+1,1} \revk{1} \ea_{i,2}  \; \sym{\eb_{i,2}},$$
a restatement of  $\sss\ii \ss{i+1} \revk{1} \ss{i+1}\ss{i}\sss{i+1}\sss{i}$.
Assume $p \ge 2$. Applying the induction hypothesis once, plus one
reversing step of type~I and one step of type~III---or one type~I'
step instead---and $2p-1$~steps of type~II, we obtain
\begin{alignat*}{4}
\sym{\eb_{i,p}} \; \ea_{i+1,p} 
&=
&&\sss{i+p-1} \; \BOLD{\sym{\eb_{i,p-1}} \;  \ss{i+p}} \; \ea_{i+1,p-1}\\
&\revk{p-1}
&&\sss{i+p-1} \; \ss{i+p} \; \BOLD{\sym{\eb_{i,p-1}} \; \ea_{i+1,p-1}}\\
&\revk{\Nba_{\pp-1}}\ 
&&\BOLD{\sss{i+p-1} \; \ss{i+p}}  \; \ea_{i,p} \; \sym{\eb_{i,p}}\\
&\revk1
&&\ss{i+p} \; \ss{i+p+1} \; \sss{i+p} \; \BOLD{\sss{i+p+1} \; \ea_{i+1,p}} \;
\sym{\eb_{i,p}}\\ 
&\revk0
&&\ss{i+p} \; \ss{i+p+1} \; \BOLD{\sss{i+p} \; \ea_{i+1,p-1}} \;
\sym{\eb_{i,p}}\\ 
&\revk{p}
&&\ss{i+p} \; \ss{i+p+1} \; \ea_{i+1,p-1} \; \sss{i+p}\; \sym{\eb_{i,p}}
= \ea_{i,p+1}  \; \sym{\eb_{i,p+1}},
\end{alignat*}
where, in each case, the factors that are about to be reversed are marked in bold. We
deduce $\Nba_{p}=\Nba_{p-1}+2p = p^2+p-1$.

The computation for~\eqref{E:DC} is illustrated in Figure~\ref{F:Quadratic}. Using
\eqref{E:BA} four times, plus
$4p+1$ type~II steps, we obtain:
\begin{alignat*}{4}
\sym{\ed_{i,p}}  \; \ec_{i+2,p}
&=
&&\sym{\eb_{i,p}} \; \BOLD{\sym{\eb_{i+1,p}} \; \ea_{i+2,p}} \; \ea_{i+3,p} \\  
&\revk{\Nba_p}  
&&\sym{\eb_{i,p}}  \; \ea_{i+1,p+1}  \; \sym{\eb_{i+1,p+1}} \;
\ea_{i+3,p}\\  
&= 
&&\BOLD{\sym{\eb_{i,p}} \; \ss{i+p+1}}  \; \ea_{i+1,p}   \; \sym{\eb_{i+1,p+1}}
\; \ea_{i+3,p}\\  
& \revk{p}\ 
&&\ss{i+p+1}  \; \BOLD{\sym{\eb_{i,p}} \; \ea_{i+1,p} }  \; \sym{\eb_{i+1,p+1}}
\; \ea_{i+3,p}\\ 
& \revk{\Nba_{p}}\ 
&&\ss{i+p+1}  \; \ea_{i,p+1}  \; \sym{\eb_{i,p+1}}  \; 
\sym{\eb_{i+1,p+1}}  \; \ea_{i+3,p}\\
&= 
&&\ea_{i,p+2}  \; \sym{\eb_{i,p+1}}  \; 
\sym{\eb_{i+2,p}} \; \BOLD{\sss{\ii+1}  \; \ea_{i+3,p}}\\
& \revk{p}
&&\ea_{i,p+2}  \; \sym{\eb_{i,p+1}}  \; 
\BOLD{\sym{\eb_{i+2,p}}  \; \ea_{i+3,p}} \; \sss{\ii+1} \\
& \revk{\Nba_{p}}
&&\ea_{i,p+2}  \; \sym{\eb_{i,p+1}}  \; \ea_{i+3,p+1}
 \; \sym{\eb_{i+3,p+1}} \;  \sss{i+1}\\
&= 
&&\ea_{i,p+2}  \; \sym{\eb_{i+1,p}}  \BOLD{\; \sss{i} \; \ss{i+p+2} \;
\ea_{i+2,p}} \; \sym{\eb_{i+2,p+2}} \\
& \revk{p+1}\  
&&\ea_{i,p+2} \;  \BOLD{\sym{\eb_{i+1,p}}   \; \ss{i+p+2}} \; \ea_{i+2,p} 
 \; \sss{i}  \; \sym {\eb_{i+1,p+1}}\\
& \revk{p} 
&&\ea_{i,p+2}  \; \ss{i+p+2} \;  \BOLD{\sym{\eb_{i+1,p}}  \; \ea_{i+2,p}} 
 \; \sss{i}  \; \sym {\eb_{i+1,p+1}}\\
& \revk{\Nba_p} 
&&\ea_{i,p+2}  \; \ss{i+p+2} \;  \ea_{i+1,p+1}  \; \sym{\eb_{i+1,p+1}}
 \; \sss{i}  \; \sym{\eb_{i+1,p+1}}
= \ec_{i,p+2}  \; \sym{\ed_{i,p+2}},
\end{alignat*}
leading to $\Ndc_\pp =4p^2+8p-3$.
\end{proof}

We can now establish Proposition~\ref{P:Quartic}.

\begin{proof}[Proof of Proposition~\ref{P:Quartic}]
(See Figure~\ref{F:Quartic}.)
We put
$$u_{\ell}=\ss{2\ell}\ss{2\ell-2} ... \ss{2}
\mbox{\quad and \quad}
v_{\ell}=\ss{1}\ss{3} ... \ss{2\ell-1},$$
and analyze the reversing of $\sym u_{\ell}v_{\ell}$.
The latter consists of three sequences of elementary
steps. First, $\ell(\ell-2)/2$ steps of type~II lead to 
$$\sss2 \; \ss1 \; \sss4 \;  \ss3 \; ... \; \sss{2\ell} \; \ss{2\ell-1}.$$
Then, $\ell$ type~I steps lead to
$\ss1\ss2\sss1\sss2 \; \ss3\ss4 \sss3\sss4 \; ...\;
\ss{2\ell-1}\ss{2\ell}\sss{2\ell-1}\sss{2\ell}$, which is
$$\ec_{1,1} \; \sym{\ed_{2,1}} \; 
\ec_{3,1} \; \sym{\ed_{4,1}} \;  ... \; 
\ec_{2\ell-1,1} \; \sym{\ed_{2\ell,1}}.$$
From there, we apply~\eqref{E:DC} repeatedly: after $\ell-1$ applications, we obtain
$$\ec_{1,1} \cdot \ec_{1,3} \; \sym{\ed_{2,3}} \; 
\ec_{3,3} \; \sym{\ed_{4,3}} \;  ... \; 
\ec_{2\ell-3,3} \; \sym{\ed_{2\ell-2,3}}  \cdot 
\sym{\ed_{2\ell,1}};$$
after $\ell-2$ more applications, we obtain
$$\ec_{1,1} \; \ec_{1,3}  \cdot  \ec_{1,5} \; \sym{\ed_{2,5}} \; 
\ec_{3,5} \; \sym{\ed_{4,3}} \;  ... \; 
\ec_{2\ell-5,5} \; \sym{\ed_{2\ell-4,5}}  \cdot  \sym{\ed_{2\ell-2,3}} \;
\sym{\ed_{2\ell,1}},$$
and so on. After using~\eqref{E:DC} $\ell(\ell-1)/2$~times, we finally obtain 
$$\ec_{1,1} \; \ec_{1,3} ... \ec_{1, 2\ell+1} \; \sym{\ed_{2,2\ell-1}} \; 
\sym{\ed_{4,2\ell-3}} \;  ... \; 
\sym{\ed_{2\ell-2,3}} \;
\sym{\ed_{2\ell,1}}.$$
A careful bookkeeping shows that the total number of reversing steps involved in the
process is $(8\ell^4-23\ell^2+9\ell+12)/6$, hence $\Theta(\ell^4)$ as announced.
\end{proof}

\begin{figure}[htb]
\begin{picture}(60,42)(0,-1)
\put(-0.3,0){\includegraphics{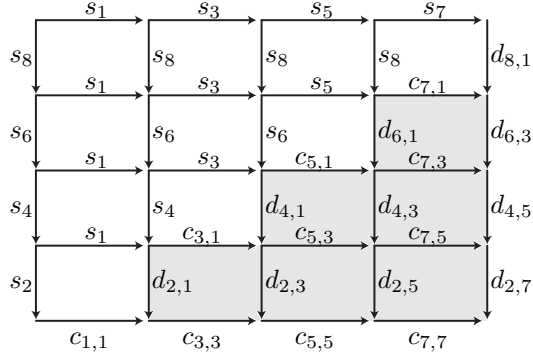}}
\put(7,42){$\ss1$}
\put(22,42){$\ss3$}
\put(37,42){$\ss5$}
\put(52,42){$\ss7$}

\put(7,32){$\ss1$}
\put(22,32){$\ss3$}
\put(37,32){$\ss5$}
\put(50,32){$\ec_{7,1}$}

\put(7,22){$\ss1$}
\put(22,22){$\ss3$}
\put(35,22){$\ec_{5,1}$}
\put(50,22){$\ec_{7,3}$}

\put(7,12){$\ss1$}
\put(20,12){$\ec_{3,1}$}
\put(35,12){$\ec_{5,3}$}
\put(50,12){$\ec_{7,5}$}

\put(5,-2){$\ec_{1,1}$}
\put(20,-2){$\ec_{3,3}$}
\put(35,-2){$\ec_{5,5}$}
\put(50,-2){$\ec_{7,7}$}

\put(-3,5.5){$\ss2$}
\put(-3,15.5){$\ss4$}
\put(-3,25.5){$\ss6$}
\put(-3,35.5){$\ss8$}

\put(16,5.5){$\ed_{2,1}$}
\put(16,15.5){$\ss4$}
\put(16,25.5){$\ss6$}
\put(16,35.5){$\ss8$}

\put(31,5.5){$\ed_{2,3}$}
\put(31,15.5){$\ed_{4,1}$}
\put(31,25.5){$\ss6$}
\put(31,35.5){$\ss8$}

\put(46,5.5){$\ed_{2,5}$}
\put(46,15.5){$\ed_{4,3}$}
\put(46,25.5){$\ed_{6,1}$}
\put(46,35.5){$\ss8$}

\put(61,5.5){$\ed_{2,7}$}
\put(61,15.5){$\ed_{4,5}$}
\put(61,25.5){$\ed_{6,3}$}
\put(61,35.5){$\ed_{8,1}$}
\end{picture}
\caption{\sf Proof of Proposition~\ref{P:Quartic}, here for $\ell = 4$; each grey rectangle
corresponds to applying Relation~\eqref{E:DC}, hence contains a number of elementary
tiles that lies in~$O(\ell^2)$.}
\label{F:Quartic}
\end{figure}

\begin{rema}
\label{R:Optimal}
At the expense of using one type~I' tile for the proof of
Relation~\ref{E:BA}, no type~III tile is used thoughout the above constructions. Using
Proposition~\ref{P:OptimalBis}, we conclude that the reversing diagram we obtained
gives an optimal van Kampen diagram, \ie, it realizes the combinatorial distance between the
boundary words, here $\uu_\ell \, \ec_{1,1} \, \ec_{1,3}\, ... \,\ec_{1, 2\ell+1}$ and
$\vv_\ell \, \ed_{2\ell, 1} \, \ed_{2\ell-2, 3} \, ..., \ed_{4,2\ell-3} \, \ed_{2,2\ell-1}$.
\end{rema}

With Proposition~\ref{P:Quartic}, we prove that~$\compl(u,v)$  can be
quartic in the length of~$u$ and~$v$. We conjecture this lower bound is
also an upper bound, but have no proof of this result so far. 
The problem is that we have no
control on the number of hexagons and digons that may
occur in a reversing diagram. There is a quadratic upper
bound on the lengths of the final expressions~$\uu',
\vv'$ that may arise from some initial pair~$(\uu,\vv)$ of length~$\ell$
expressions, but this does not directly lead to a bound on the number of
type~I reversing steps used (the only ones that increase the length)
because some subsequent type~III steps might erase the letters so created.

\begin{rema}
As mentioned in the introduction, most results of this paper extend to positive braids. For
instance, the optimality criterion of Proposition~\ref{P:MinimalSep} extends to positive
braids at the expense of adding a notion of rank in the definition of separatrices: in the
braid diagram associated with a permutation, \ie, with a simple braid, any two strands cross
at most once, and we introduce one separatrix~$\SEP\pp\qq$ only. For the case of
arbitrary positive braids, we should introduce several separatrices for pairs of strands that
cross more than one time, $\SEP\pp\qq^{(\kk)}$ being associated with the $\kk$th
intersection of the strands~$\pp$ and~$\qq$. As for subword reversing, it works in the
general braid case exactly as in the case of simple braids, \ie, of permutations. Experiments
show that the worst cases in terms of complexity arise with simple braids. So we
have no better result in the general braid case than in the particular permutation case.
\end{rema}

\end{document}